\documentclass[letterpaper, 10pt, conference]{ieeeconf}
\IEEEoverridecommandlockouts 
\usepackage{amsmath,amssymb,url,times,subfigure,graphicx,theorem}
\usepackage{graphicx,subfigure,hyperref}
\usepackage{color,comment}

\newcommand{\bracket}[1]{\ensuremath{\left[ #1 \right]}}
\newcommand{\braces}[1]{\ensuremath{\left\{ #1 \right\}}}
\newcommand{\parenth}[1]{\ensuremath{\left( #1 \right)}}

\newcommand{\refeqn}[1]{(\ref{eqn:#1})}
\newcommand{\reffig}[1]{Fig. \ref{fig:#1}}
\newcommand{\tr}[1]{\mbox{tr}\ensuremath{\negthickspace\bracket{#1}}}
\newcommand{\trs}[1]{\mbox{tr}\ensuremath{\!\bracket{#1}}}

\newcommand{\SO}{\ensuremath{\mathsf{SO(3)}}}

\newcommand{\so}{\ensuremath{\mathfrak{so}(3)}}

\renewcommand{\Re}{\ensuremath{\mathbb{R}}}
\newcommand{\Sph}{\ensuremath{\mathsf{S}}}

\title{\LARGE \bf
Global Unscented Attitude Estimation via\\ the Matrix Fisher Distributions on $\SO$}

\author{Taeyoung Lee%
\thanks{Taeyoung Lee, Mechanical and Aerospace Engineering, George Washington University, Washington DC 20052 {\tt tylee@gwu.edu}}
\thanks{This research has been supported in part by NSF under the grants CMMI-1243000, CMMI-1335008, and CNS-1337722.}
}

\newtheorem{definition}{Definition}

\newtheorem{prop}{Proposition}

\graphicspath{{./Figs/}}

\begin{document}
\allowdisplaybreaks

\maketitle \thispagestyle{empty} \pagestyle{empty}

\begin{abstract}
This paper is focused on probabilistic estimation for the attitude dynamics of a rigid body on the special orthogonal group. We select the matrix Fisher distribution to represent the uncertainties of attitude estimates and measurements in a global fashion without need for local coordinates. Several properties of the matrix Fisher distribution on the special orthogonal group are presented, and an unscented transform is proposed to approximate a matrix Fisher distribution by selected sigma points. Based on these, an intrinsic, global framework for Bayesian attitude estimation is developed. It is shown that the proposed approach can successfully deal with large initial estimator errors and large uncertainties over complex maneuvers to obtain accurate estimates of the attitude. 
\end{abstract}

\section{Introduction}


Attitude estimation has been widely studied with various filtering approaches and assumptions~\cite{CraMarJGCD07}. One of the biggest challenges is that the attitude dynamics evolve on a compact, nonlinear manifold, namely the special orthogonal group. The attitude is often parameterized by certain three dimensional coordinates, and an estimator is developed in terms of these local coordinates. However, it is well known that minimal, three-parameter attitude representations, such as Euler-angles or modified Rodriguez parameters, suffer from singularities. They are not suitable for large angle rotational maneuvers, as the type of parameters should be switched persistently in the vicinity of singularities. 


Quaternions are another popular choice in attitude estimation~\cite{CraMarJGCD97,PsiJGCD00}. They do not exhibit singularities, but as the configuration space of quaternions, namely the three-sphere double covers the special orthogonal group, there exists ambiguity. More explicitly, a single attitude may be represented by two antipodal points on the three-sphere. The ambiguity should be carefully resolved in any quaternion-based attitude observer and controller, otherwise they may exhibit unwinding, for example~\cite{BhaBerSCL00}. Furthermore, quaternions are often considered as vectors in $\Re^4$, instead of incorporating the structures of the three-sphere carefully when designing attitude estimators. 

Instead, attitude observers have been designed directly on the special orthogonal group to avoid both singularities of local coordinates and the ambiguity of quaternions. The development for \textit{deterministic} attitude observers on the special orthogonal group includes complementary filters~\cite{MahHamITAC08}, a robust filter~\cite{SanLeeSCL08}, and a global attitude observer~\cite{WuKauPICDC15}.

The prior efforts to construct \textit{probabilistic} attitude estimators on the special orthogonal group and the relevant research have been relatively unpopular compared with deterministic approaches, especially in the engineering community. Probability and stochastic processes on manifolds have been studied in~\cite{Eme89,Elw82}. Directional statistics have been applied in earth sciences and material sciences~\cite{MarJup99,Chi03}.

Earlier works on attitude estimation on the special orthogonal group include~\cite{LoEshSJAM79}, where a probability density function is expressed using noncommutative harmonic analysis~\cite{ChiKya01}. This idea of using Fourier analysis on manifolds has been applied for uncertainty propagation and attitude estimation~\cite{ParKimP2IICRA05,MarPMDSAS05,LeeLeoPICDC08}. The use of noncommutative harmonic analysis allows a probability density function to be expressed globally, and the Fokker-Plank equation to be transformed into ordinary differential equations, thereby providing a fundamental solution for the Bayesian attitude estimation. However, in practice they may  cause computational burden, since a higher order of Fourier transform is required as the estimated distribution becomes more concentrated. 

Recent literature is rich with filtering techniques and measurement models developed in terms of exponential coordinates~\cite{ParLiuR08,Chi12,ChiKobPICDC14,LonWolRSV13}. This is perhaps the most natural approach to develop an estimator formally on an abstract Lie group, while taking advantages of the fact that the lie algebra is a linear space. The limitation is that the exponential map is a local diffeomorphism around the identity element, and as such, the issue of a singularity remains.

This paper aims to construct a probabilistic attitude estimator on the special orthogonal group, while avoiding complexities of harmonic analysis and singularities of exponential coordinates. We use a specific form of the probability density, namely the matrix Fisher distribution~\cite{MarJup99}, to represent uncertainties in the estimates of attitudes. Therefore, the proposed approach can be considered as an example of \textit{assumed density filtering}. To project the propagated density onto the space of the matrix Fisher distributions, an unscented transform and its inverse are proposed. Assuming that the attitude measurement errors are represented by a matrix Fisher distribution, it is shown that the posteriori estimation also follows the Fisher distribution.

These provide a Bayesian, probabilistic attitude estimator on the special orthogonal group in a global fashion. It is demonstrated that the proposed estimator exhibits excellent convergence properties even with large initial estimation errors and large uncertainties, in contrast to the attitude estimators based on local coordinates and linearization that tend to diverge for such challenging cases. 


\section{Matrix Fisher Distribution on $\SO$}\label{sec:MF}

Directional statistics deals with statistics for unit-vectors and rotations in $\Re^n$, where various  probability distributions on nonlinear compact manifolds are defined, and statistical analysis, such as inference and regressions are studied in those manifolds~\cite{MarJup99,Chi03}. In particular, the matrix Fisher (or von Mises-Fisher matrix) distribution is a simple exponential model introduced in~\cite{DowB72,KhaMarJRSSS77}. Interestingly, many of the prior work on the matrix Fisher distributions in directional statistics are developed for the Stiefel manifold, $\mathsf{V}_k(\Re^n)=\{X\in\Re^{n\times k}\,|\, XX^T=I_{n\times n}\}$.

The configuration manifold for the attitude dynamics of a rigid body is the three-dimensional special orthogonal group, 
\begin{align*}
\SO = \{R\in\Re^{3\times 3}\,|\, R^TR=I_{3\times 3},\,\mathrm{det}[R]=1\}.
\end{align*}
This section provides the definition of the matrix Fisher distribution and several properties developed for $\SO$. 

Throughout this paper, the \textit{hat} map: $\wedge:\Re^3\rightarrow \so$ is defined such that $\hat x = -(\hat x)^T$, and $\hat x y =x\times y$ for any $x,y\in\Re^3$. The inverse of the hat map is denoted by the \textit{vee} map: $\vee:\so\rightarrow\Re^3$. The set of circular shifts of $(1,2,3)$ is defined as $\mathcal{I}=\{(1,2,3),(2,3,1),(3,1,2)\}$. 

\subsection{Matrix Fisher Distribution on $\SO$}

The probability density of the matrix Fisher distribution on $\SO$ is given by
\begin{align}
p(R)=\frac{1}{c(F)}\exp(\trs{F^T R}),\label{eqn:MF}
\end{align}
where $F\in\Re^{3\times 3}$ is a matrix parameter, and $c(F)\in\Re$ is a normalizing constant defined as
\begin{align}
c(F) = \int_{\SO} \exp(\trs{F^T R}) dR. \label{eqn:cF}
\end{align}
For $\SO$, there is a bi-invariant measure, referred to as \textit{Haar} measure, that is unique up to scalar multiples~\cite{ChiKya01}. The above expression is assumed to be defined with respect to the particular Haar measure $dR$ that is normalized such that $\int_{\SO} dR=1$. In other words, the uniform distribution on $\SO$ is given by $1$ with respect to $dR$. This is often stated that \refeqn{MF} is defined with respect to the uniform distribution. 
When $R$ is distributed according to the matrix Fisher distribution with the parameter matrix $F$, it is denoted by $R\sim\mathcal{M}(F)$. 

The singular value decomposition of $F$ is given by
\begin{align}
F= U S V^T,
\end{align}
where $U,V\in\Re^{3\times 3}$ are orthonormal matrices, and $S=\mathrm{diag}[s_1,s_2,s_3]$ for the singular values $s_i>0$ and $i\in\{1,2,3\}$. Throughout this paper, we assume that $\mathrm{det}[F]>0$, such that $\mathrm{det}[U]\mathrm{det}[V]=1>0$. Then, $U,V\in\SO$ holds without loss of generality (in the case $\mathrm{det}[U]=\mathrm{det}[V]=-1$, we can multiply $U,V$ by $-1$).

Let $K\in\Re^{3\times 3}$ and $M\in\SO$ be the elliptic component and the polar component of $F$, i.e.,
\begin{align}
F=KM,\quad K=K^T=USU^T,\quad M=UV^T.\label{eqn:KM}
\end{align}
Since $\trs{F^T R} = \trs{VSU^T R} = \trs{S U^T RV}$, the probability density $p(R)$ is maximized when $R=M$ for a fixed $F$. Therefore, the polar component $M$ is considered as the \textit{mean} attitude. The matrices $S$ and $U$ of the elliptic component determine the degree and the direction of dispersion about the mean attitude. More specifically, the probability density becomes more concentrated as the singular value $s_i$ increases. The role of $S,U$ in determining the shape of the distribution will be discussed more explicitly at Section \ref{sec:UAE}.


While there are various approaches to evaluate the normalizing constant for the matrix Fisher distribution on the Stiefel manifold, only a few papers deal with the normalizing constant on $\SO$. A method based on the holonomic gradient descent is introduced in~\cite{SeiShiJMA13}, which involves the numerical solution of multiple ordinary differential equations. The normalizing constant is expressed as a simple one-dimensional integration in~\cite{WooAJS93}, but the given result is erroneous as the change of volume over a certain transformation is not considered properly. We follow the approach of~\cite{WooAJS93}, to find a closed form expression of the normalizing constant.

\begin{prop}
The normalizing constant for the matrix Fisher distribution \refeqn{MF} is given by
\begin{align}
c(F) = c(S) & = \int_{-1}^1 \frac{1}{2}I_0\!\bracket{\frac{1}{2}(s_i-s_j)(1-u)} \nonumber\\
&\times I_0\!\bracket{\frac{1}{2}(s_i+s_j)(1+u)}\exp (s_ku)\,du,\label{eqn:cS}
\end{align}
where $(i,j,k)\in\mathcal{I}$, and $I_0$ denotes the zero degree, modified Bessel function for the first kind~\cite{AbrSte65}, i.e., $I_0(u) = \sum_{r=0}^\infty (\frac{1}{2}u)^{2r}/(r!)^2$.
\end{prop}
\begin{proof}
See Appendix \ref{sec:PfC}.
\end{proof}
This implies that the normalizing constant only depends on $S$, and the order of the singular values in $S$ can be shifted. It is not burdensome to evaluate \refeqn{cS} numerically, as it takes less than 0.01 second with the 2.4 GHz Intel Core i5 processor in Matlab. Also, from \refeqn{cS}, it is straightforward to find a closed form of the derivatives of the normalizing constant with respect to $s_i$, which is useful for maximum log-likelihood estimation of the matrix parameter~\cite{DowB72,KhaMarJRSSS77}. 

\subsection{Visualization of the matrix Fisher distribution}

A method to visualize any probability density function on $\SO$ has been proposed in~\cite{LeeLeoPICDC08}. Let $r_i\in\Sph^2$ be the $i$-th column of a rotation matrix $R$, i.e., $R=[r_1,r_2,r_3]\in\SO$, where the two-sphere is the space of unit-vectors in $\Re^3$, i.e., $\Sph^2=\{q\in\Re^3\,|\, \|q\|=1\}$.  The key idea for visualization is that $r_i$ has a certain geometric meaning of the attitude, namely the direction of the $i$-th body-fixed frame in the inertial frame. Once the marginal distribution for $r_i$ is obtained from a probability density function of $\SO$, it can be visualized on the surface of the unit-sphere via color shading. If the distribution of each $r_i$ is mildly concentrated, the distributions of all of three body-fixed axes can be visualized at the single unit-sphere, thereby illustrating the shape of attitude probability dispersion intuitively.

Here we show that the marginal distribution for the matrix Fisher distribution can be obtained in a closed form.

\begin{prop}
Suppose $R\sim\mathcal{F}(M)$. Let $(i,j,k)\in\mathcal{I}$, and let $r_i\in\Sph^2$ be the $i$-th column of $R$. Then, the marginal probability density of $r_i$ is 
\begin{align}
p(r_i) & = \frac{c_2(f_{jk},r_i)}{c(S)}   \exp (f_i^Tr_i),\label{eqn:pri}
\end{align}
with respect to the uniform distribution on $\Sph^2$, where $f_i\in\Re^3$ denotes the $i$-th column of the matrix parameter $F$, and $f_{jk}=[f_j,f_k]\in\Re^{3\times 2}$. The constant $c_2(f_{jk},r_i)$ is defined as
\begin{align}
c_2(f_{jk},r_i) = I_0 \bracket{\sum_{i=1}^2 \sqrt{\lambda_l\bracket{f_{jk}^T (I_{3\times 3}-r_ir_i^T) f_{jk}}}},\label{eqn:c2}
\end{align}
where $\lambda_l[\cdot]$ denotes the $l$-th eigenvalue of a matrix. 
\end{prop}
\begin{proof}
See Appendix \ref{sec:MD}.
\end{proof}

Visualizations for selected matrix Fisher distributions constructed via \refeqn{pri} are available in \reffig{vis}.

\begin{figure}
\centerline{
	\subfigure[$F_a=5I_{3\times 3}$]{
		\includegraphics[width=0.3\columnwidth]{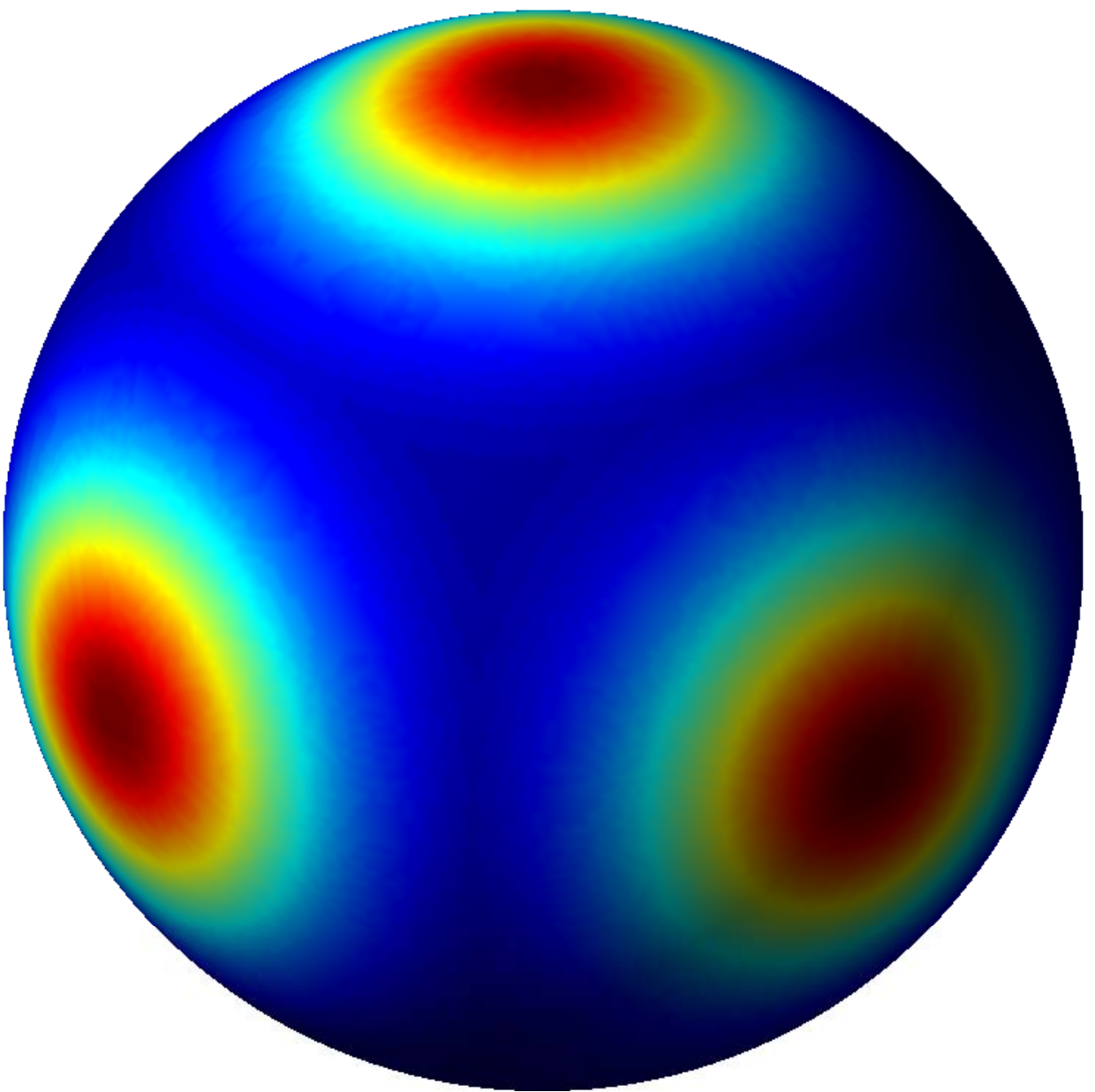}\label{fig:vis1}}
	\hfill
	\subfigure[$F_b=20I_{3\times 3}$]{
		\includegraphics[width=0.3\columnwidth]{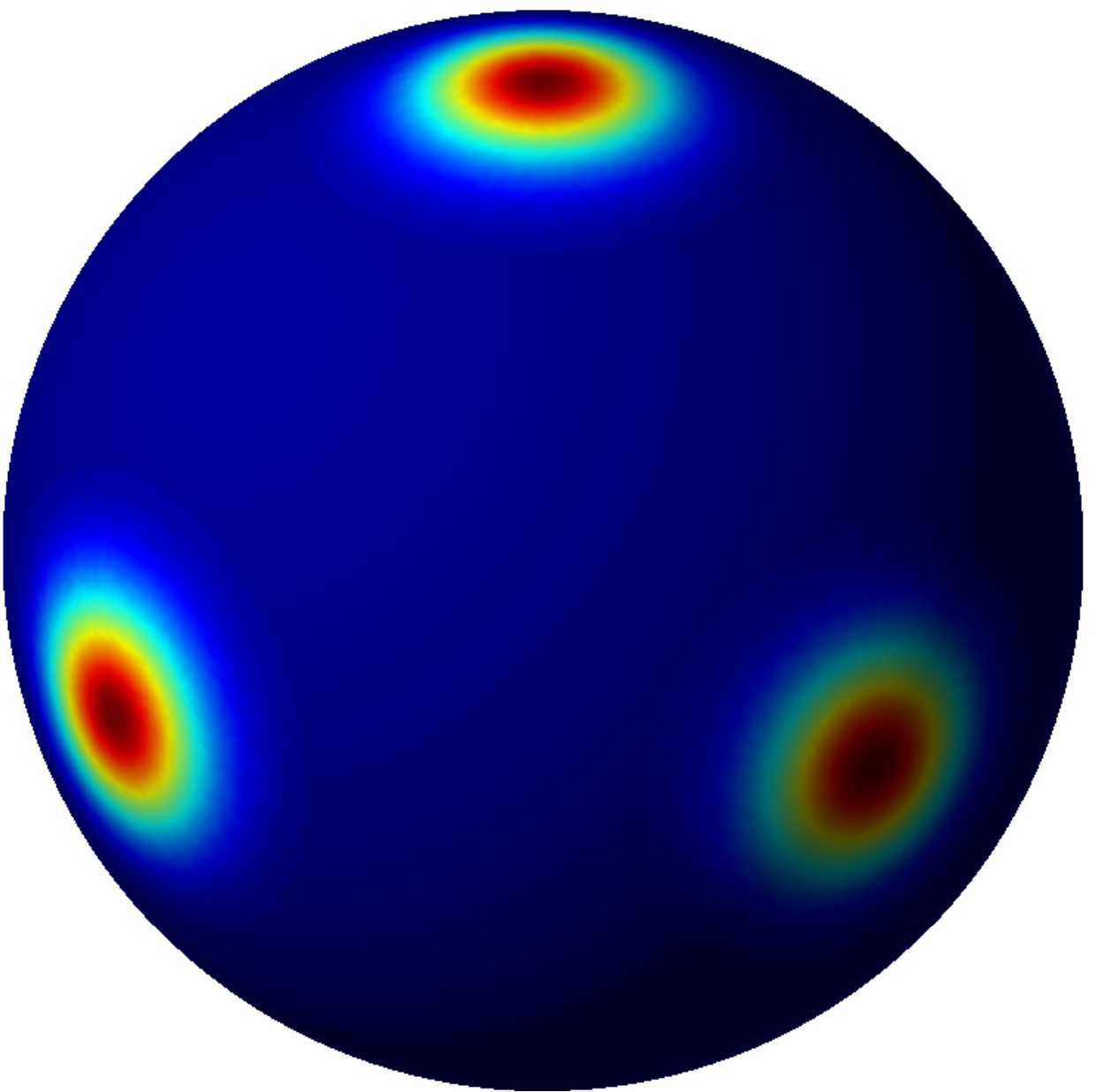}\label{fig:vis2}}
	\hfill
	\subfigure[{$F_c=\mathrm{diag}[25,5,1]$}]{\hspace*{0.005\textwidth}
		\includegraphics[width=0.3\columnwidth]{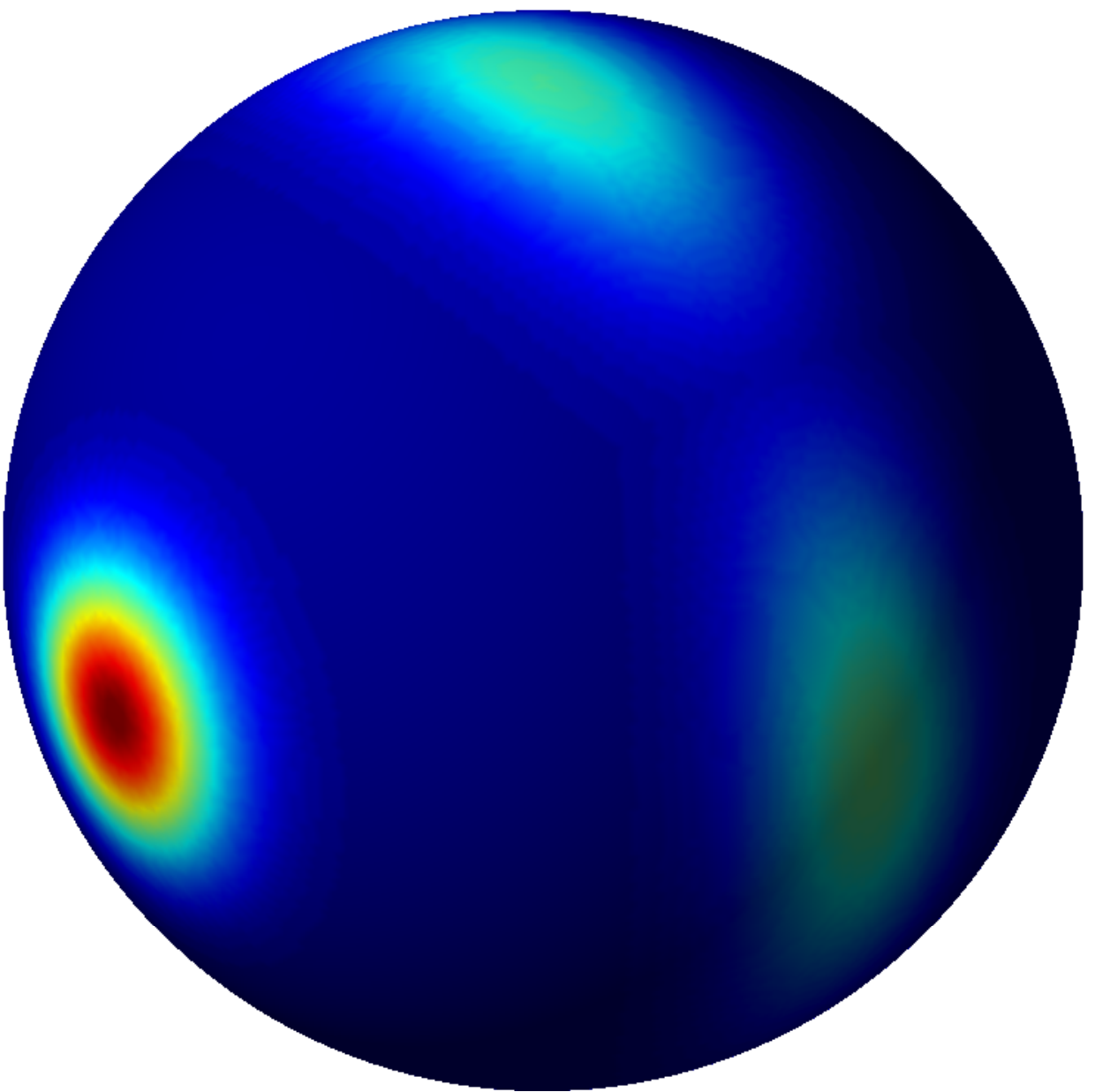}\label{fig:vis3}\hspace*{0.005\textwidth}}
}
\caption{Visualization of selected matrix Fisher distributions: the distribution in (b) is more concentrated than in (a), as the singular values of $F_b$ are greater than those of $F_a$; for both (a) and (b), the distributions of each axis are identical and circular as three singular values of each of $F_a$ and $F_b$ are identical; in (c), the first body-fixed axis (lower left) is more concentrated as the first singular value of $F_c$ is the greatest, and the distributions for the other two axes are elongated. Compared with the third body-fixed axis (top), the probability density of the second body-fixed axis (lower right) is greater, as the second singular value of $F_c$ is greater than the third.}\label{fig:vis}
\end{figure}

\section{Unscented Attitude Estimation}\label{sec:UAE}

In this section, an attitude estimation scheme is proposed based on the matrix Fisher distribution on $\SO$. Assuming that the initial attitude estimate and the attitude measurement errors are described by certain matrix Fisher distributions on $\SO$, we construct an estimated attitude distribution via another matrix Fisher distribution following a Bayesian framework. Therefore, this approach is an example of so-called, \textit{assumed density filters}. 

One issue of any assumed density filter is that the propagated uncertainty is not guaranteed to be distributed as the selected density model. This has been commonly addressed by two distinct approaches. The first one is approximating the dynamics such that the propagated uncertainty follows the selected density model. For example, in extended Kalman filters, the dynamics is linearized to ensure that the propagated uncertainty is Gaussian. The second option is instead approximating the density model by selected parameters along the solution of the exact dynamic model, such as in unscented filters. In short, selecting one of these corresponds to the following question of `what should be approximated between dynamics and probability distributions?'

In attitude estimation problems, the  equations of motion are well known, but it is often challenging to obtain accurate probability distributions. In such cases, it may be reasonable to approximate probability distributions rather than corrupting the exact dynamic model by approximations. Here, we propose an unscented transform to approximate the matrix Fisher distribution on $\SO$ by selected sigma points, and based on this, we construct a Bayesian attitude estimator.

\subsection{Unscented transform for matrix Fisher distribution}

Suppose $R\sim\mathcal{M}(F)$. We wish to define several rotation matrices that approximate $\mathcal{M}(F)$. This is achieved by identifying the role of the elliptic component and the polar component of the matrix parameter $F$ introduced in \refeqn{KM}. Consider a set of rotation matrices parameterized $\theta_i\in[0,2\pi)$ for $i\in\{1,2,3\}$ as
\begin{align}
R_i(\theta_i) = \exp(\theta_i\widehat {Ue_i}) UV^T=U\exp(\theta_i\hat e_i) V^T,\label{eqn:Ri}
\end{align}
where $e_i\in\Re^3$ denotes the $i$-th column of $I_{3\times 3}$. This corresponds to the rotation of the mean attitude $M=UV^T$ about the axis $Ue_i$ by the angle $\theta_i$, where $Ue_i$ is considered expressed with respect to the inertial frame. 

Using \refeqn{KM}, the probability density \refeqn{MF} along \refeqn{Ri} is given by
\begin{align}
p(R_i(\theta_i)) & =\frac{1}{c(F)}\exp(\trs{VSU^T  U\exp(\theta_i\hat {e}_i) V^T})\nonumber\\
&=\frac{1}{c(S)}\exp(\trs{S\exp(\theta_i\hat{e}_i)}).\label{eqn:pRi0}
\end{align} 
Substituting Rodrigues' formula~\cite{ShuJAS93}, namely $\exp(\theta_i\hat{e}_i)=I_{3\times 3}+\sin\theta_i\hat e_i +(1-\cos\theta_i)\hat e_i^2$, and rearranging,
\begin{align}
p(R_i(\theta_i)) 
&=\frac{1}{c(S)}\exp(s_i + (s_j+s_k)\cos\theta_i),\label{eqn:pRi}
\end{align} 
where $j,k$ are determined such that $(i,j,k)\in\mathcal{I}$. This resembles the von Mises distribution on a circle, where the probability density is proportional to $\exp^{\kappa\theta}$ for a concentration parameter $\kappa\in\Re$~\cite{MarJup99}.

The most noticeable property of \refeqn{pRi0} and \refeqn{pRi} is that the probability density depends only on the singular values $s_i$ and the rotation angle $\theta_i$, and it is independent of $U$ or $V$. When considered as a function of $\theta_i$, the overall value of $p(R_i(\theta_i)) $ would increase as $s_i$ becomes larger, and the curve becomes narrower as $s_j+s_k$ increases. For example, a larger $s_1$ implies that the marginal probability density of the first body-fixed axis increases, and the distributions of the marginal probability densities of the second axis and the third axis become narrower along the rotations about the third axis and the second axis, respectively, as illustrated in \reffig{vis3}. Recall \refeqn{Ri} is obtained by rotating the mean attitude $M=UV^T$ about the $i$-th column of $U$. As such, each column of $U$ is considered as the \textit{principle axis} of rotation for $\mathcal{M}(F)$. 

In short, the role of $F=USV^T$ in determining the shape of the distribution of $\mathcal{M}(F)$ is as follows: (i) the rotation matrix $U$ sets the principle axis of rotations; (ii) the singular vales $S$ describe the concentration of the distribution along the principle axes; (iii) the rotation matrix $V$ determines the mean attitude $M=UV^T$, together with $U$. 

In unscented transformations for a Gaussian distribution in $\Re^n$, the sigma points are chosen along the principle axis. Motivated by this and the above observations, the following unscented transform is proposed for the matrix Fisher distribution on $\SO$.

\begin{figure}
\centerline{
	\subfigure[$F=5I_{3\times 3}$]{
		\includegraphics[width=0.45\columnwidth]{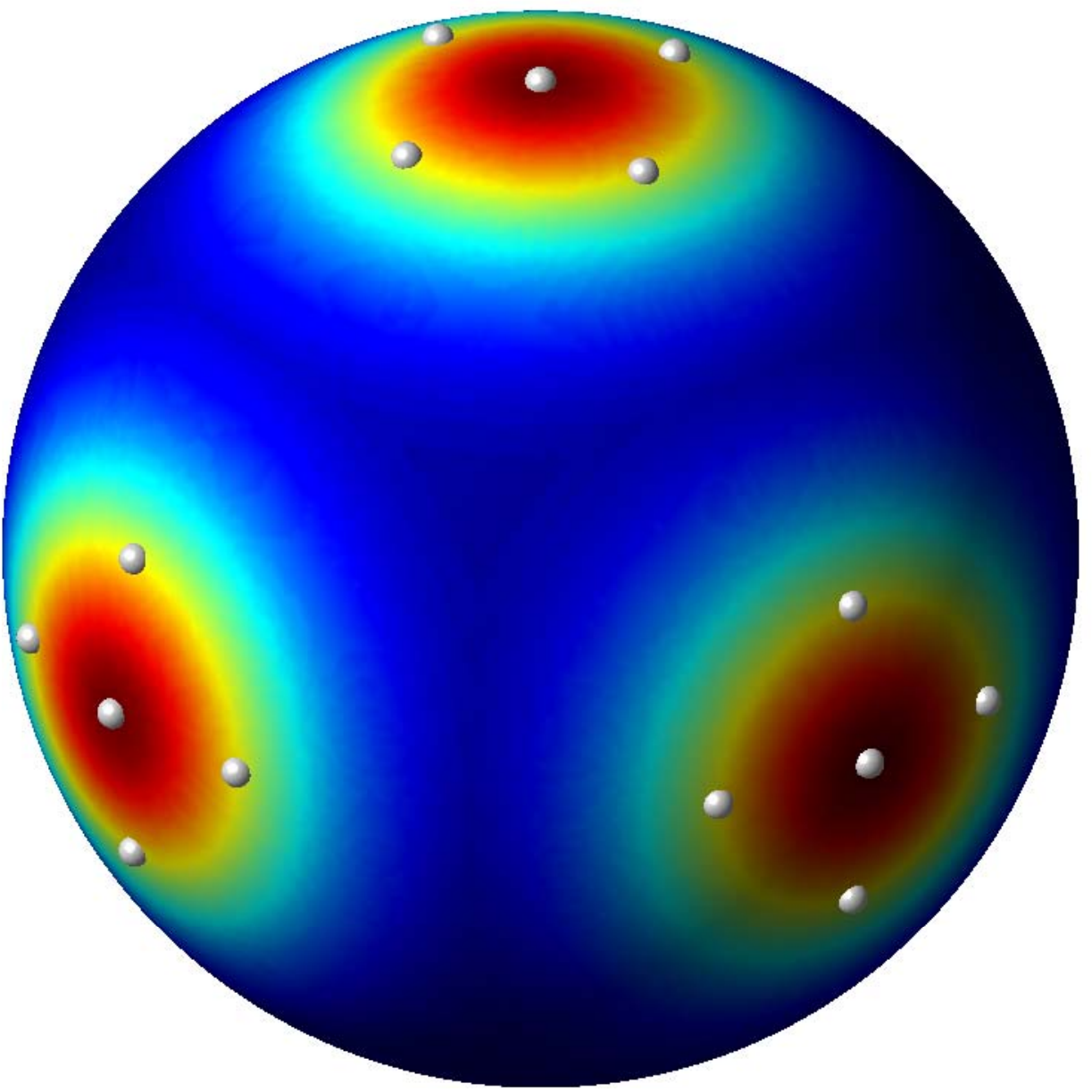}\label{fig:sig1}}
	\hfill
	\subfigure[{$F=\mathrm{diag}[25,5,1]$}]{
		\includegraphics[width=0.45\columnwidth]{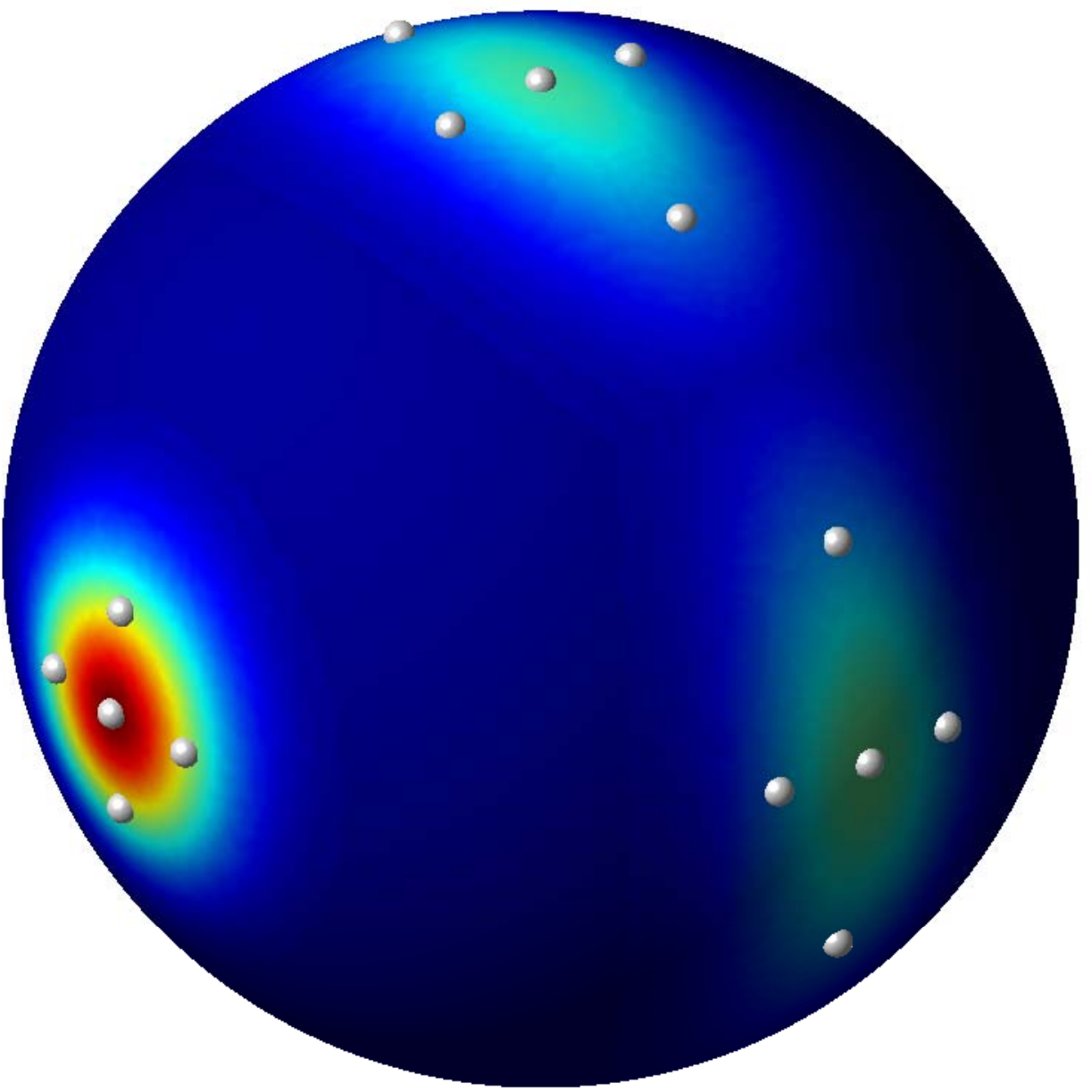}\label{fig:sig3}}
}
\caption{Visualization of sigma points: the body-fixed axes of the  sigma points selected by \refeqn{SP}, \refeqn{costhetai} with $\sigma=0.9$ are illustrated by white dots.}\label{fig:sig}
\end{figure}

\begin{definition}
Consider a matrix Fisher distribution $\mathcal{M}(F)$, and let the singular value decomposition of $F$ is given by \refeqn{KM}. The set of  seven sigma points is defined as
\begin{align}
\{M\}\cup\{\,R_i(\theta_i),R_i(-\theta_i)\,|\, i\in\{1,2,3\}\},\label{eqn:SP}
\end{align}
where each angle $\theta_i$ is chosen as
\begin{align}
\cos\theta_i = \frac{(1-\sigma)\log c(S) +\sigma s_T- s_i}{s_j+s_k},\label{eqn:costhetai}
\end{align}
for $(i,j,k)\in\mathcal{I}$. The parameter $\sigma< 1$ determines the spread of the sigma points, and $s_T=\sum_{i=1}^3 s_i$.
\end{definition}

In other words, for a given parameter matrix $F$, the seven sigma points are chosen as the mean attitude, and positive/negative rotations about each principle axis by the angle determined by \refeqn{costhetai}. Note that each sigma point corresponds to a rotation matrix in $\SO$. The equation \refeqn{costhetai} to select the rotation angle is motivated as follows. Substituting \refeqn{costhetai} into \refeqn{pRi}, and taking logarithm,
\begin{align*}
\log p(R_i(\pm \theta_i))=  \sigma(s_T-\log c(S)),
\end{align*}
for any $i\in\{1,2,3\}$. As such, the last six sigma points of \refeqn{SP} have the same value of the probability density, given by $\frac{1}{c(S)}\exp (\sigma s_T)$. The ratio of that probability density to the maximum density, $\frac{1}{c(S)}\exp(s_T)$ is given by $\exp((\sigma-1)s_T)$. 

Therefore, the last six sigma points will be closer to the mean attitude, when the distribution is concentrated with larger $s_i$, or $\sigma$ becomes larger. As $\sigma\rightarrow 1$, all of the sigma points converge to the mean attitude. 
The sigma points for selected distributions are illustrated at \reffig{sig}. 

Next, we show that the set of sigma points is statistically sufficient.

\begin{prop}\label{prop:IUT}
Suppose the seven sigma points defined in \refeqn{SP} and the parameter $\sigma$ are given for $\mathcal{M}(F)$. Let $\bar R\in\Re^{3\times 3}$ be the arithmetic mean of the sigma points, i.e.,
\begin{align}
\bar R = \frac{1}{7}\bracket{M + \sum_{i=1}^3\braces{R_i(\theta_i)+ R_i(-\theta_i)}}.\label{eqn:barR}
\end{align}
Then, the singular value decomposition of $\bar R$ is given by
\begin{align}
\bar R = U D V^T,\label{eqn:UDV}
\end{align}
where $U,V\in\SO$ corresponds to those of \refeqn{KM}, and $D=\mathrm{diag}[d_1,d_2,d_3]\in\Re^{3\times 3}$. For $(i,j,k)\in\mathcal{I}$, $d_i$ is given by
\begin{align}
d_i = \frac{1}{7}(3 + 2(\cos\theta_j+\cos\theta_k)).\label{eqn:di}
\end{align}
\end{prop}
\begin{proof}
See Appendix \ref{sec:UT}.
\end{proof}

Therefore, for given sigma points, one can find the corresponding matrix parameter $F$ as follow: (i) the matrices $U,V$ are obtained from \refeqn{UDV}; (ii) solve \refeqn{di} for $(\cos\theta_1,\cos\theta_2,\cos\theta_3)$, which can be used to determine $(s_1,s_2,s_3)$ from \refeqn{costhetai}; (iii) $F=USV^T$. 

Based on the proposed unscented transform and its inverse, we construct a Bayesian estimator as follows. 

\subsection{Unscented Attitude Estimation}


Consider a stochastic differential equation on $\SO$,
\begin{align}
(R^T dR)^\vee = \Omega_z + w_\Omega,\label{eqn:SDE}
\end{align}
where $\Omega_z,w_\Omega\in\Re^3$ are the measured angular velocity and the angular velocity measurement error, respectively. It is assumed that the value of $\Omega_z$ is provided by an angular velocity sensor. The measurement error $w_\Omega$ is random, but its distribution is known. Suppose that the attitude is also measured by a sensor, such as an inertial measurement unit, and the attitude measurement $R_z\in\SO$ is given by
\begin{align}
R_z = R W_R,\label{eqn:Rz}
\end{align}
where $W_R\in\SO$ is an attitude measurement error, and $W_R\sim\mathcal{M}(F_z)$ for a known matrix parameter $F_z\in\Re^{3\times 3}$. 

Consider a discrete time sequence $\{t_0,t_1,\ldots\}$. The attitude estimation problem considered in this paper is to find the matrix parameter $F_{k+1}$ that approximates the estimated attitude distribution at $t=t_{k+1}$ via $\mathcal{M}(F_{k+1})$ for given $F_k$, $R_{z_k}$ and $\Omega_{z_k}$ with the assumption that $R_k\sim\mathcal{M}(F_k)$. Here, the subscript $k$ denotes the value of a variable at $t=t_k$. 

The proposed estimator is composed of a propagation step and a measurement update step. 

\paragraph{Propagation} The propagation step is defined via the unscented transform as follows.
\begin{itemize}
\item[(i)] Given $F_k$, seven sigma points at $t=t_k$, namely $R^l_k$ for $l\in\{1,\ldots,7\}$ are computed via \refeqn{SP}.
\item[(ii)] Each sigma point is propagated to $t=t_{k+1}$ according to \refeqn{SDE}. For example, a second order Lie group method~\cite{HaiLub00} can be applied to obtain
\begin{align}
R^l_{k+1} = R^l_k \exp\parenth{\frac{1}{2}h(\Omega_{z_k}+w_{\Omega_k}+\Omega_{z_{k+1}}+w_{\Omega_{k+1}})},
\end{align}
where $h=t_{k+1}-t_k$ is the time step. The angular velocity measurements $\Omega_z$ are from the sensor, and the measurement errors $w_\Omega$ are sampled from the given distribution.
\item[(iii)] Find $F_{k+1}$ from the propagated sigma points $R^l_{k+1}$ according to the results of Proposition \ref{prop:IUT}.
\end{itemize}
These steps are repeated until an attitude measurement is available. 

\paragraph{Measurement Update} Suppose that the attitude is measured at $t_{k+1}$. We wish to find the distribution for $R_{k+1}|R_{z_{k+1}}$. From now on, in this subsection, we do not specify the subscript $k+1$ for brevity.
Since $W_R\sim\mathcal{M}(F_z)$, \refeqn{Rz} implies 
\begin{align}
p(R_z|R) = \frac{1}{c(F_z)}\exp(\trs {F_z^T R^T R_z}), \label{eqn:pRzR}
\end{align}
where we have used $c(RF_z)=c(F_z)$. According to Bayes' rule, the posterior distribution is
\begin{align*}
p(R|R_z) &= \frac{1}{a}p(R_z|R)P(R),
\end{align*}
where $a$ is a normalizing constant independent of $R$. Since $R\sim\mathcal{M}(F)$, from \refeqn{pRzR},
\begin{align*}
p(R|R_z) &= \frac{1}{a c(F_z)} \exp(\trs {F_z^T R^T R_z}+\trs{F^T R}).\\
&= \frac{1}{c(F+ZF_z)} \exp(\trs {(F+R_zF_z^T)^T R}).
\end{align*}
Therefore, the posterior distribution for $R_{k+1}$ also follows a matrix Fisher distribution, i.e., 
$R_{k+1}\sim \mathcal{M}(F_{k+1}+R_{z_{k+1}}F_z^T)$.

\section{Numerical Example}

\begin{figure}
\centerline{
	\subfigure[True angular velocity: $\Omega_{true}(t)$ ($\mathrm{rad/s}$)]{
		\includegraphics[width=0.45\columnwidth]{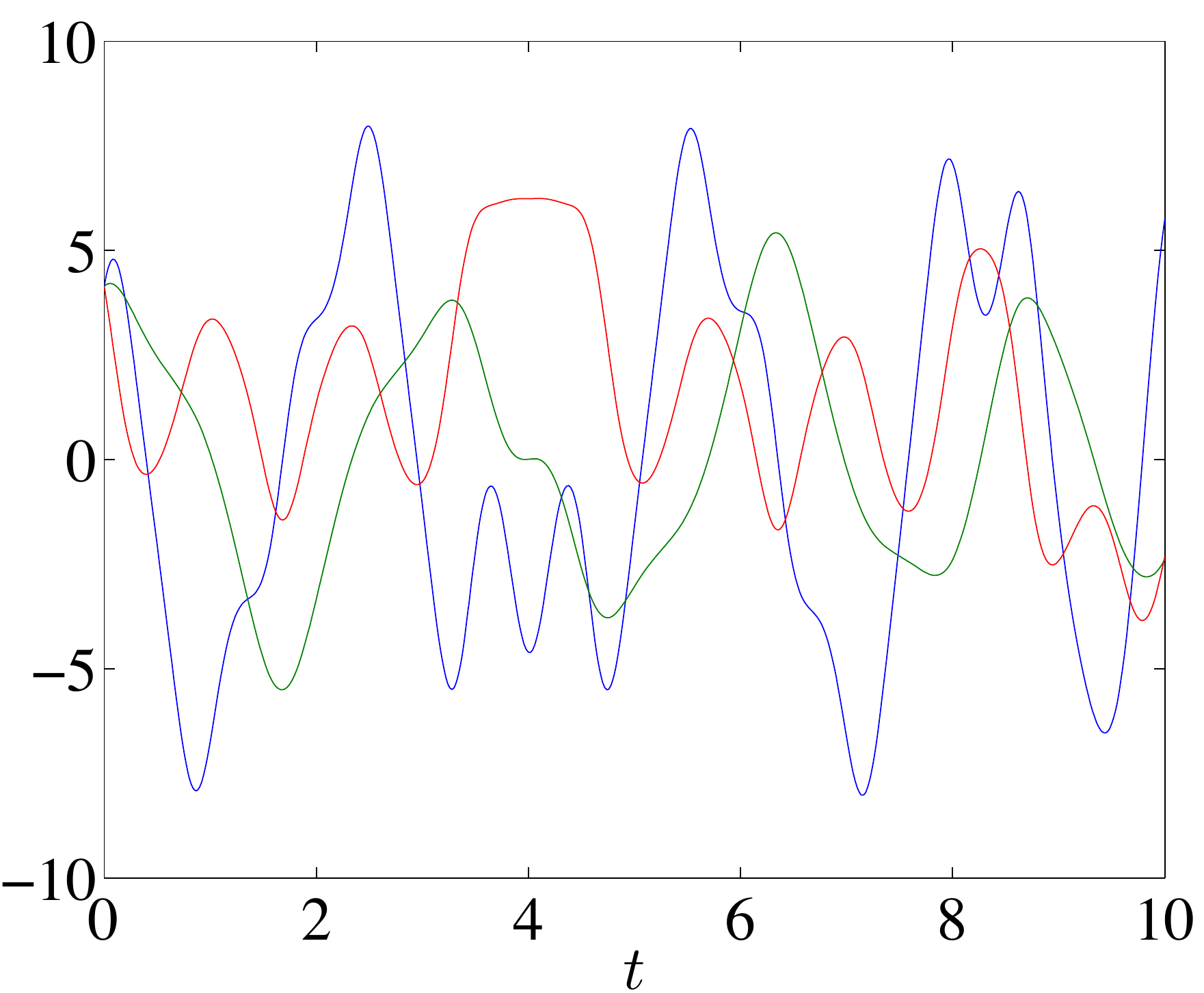}\label{fig:W_true}}
	\hfill
	\subfigure[Measured angular velocity: $\Omega_z(t)$ ($\mathrm{rad/s}$)]{
		\includegraphics[width=0.45\columnwidth]{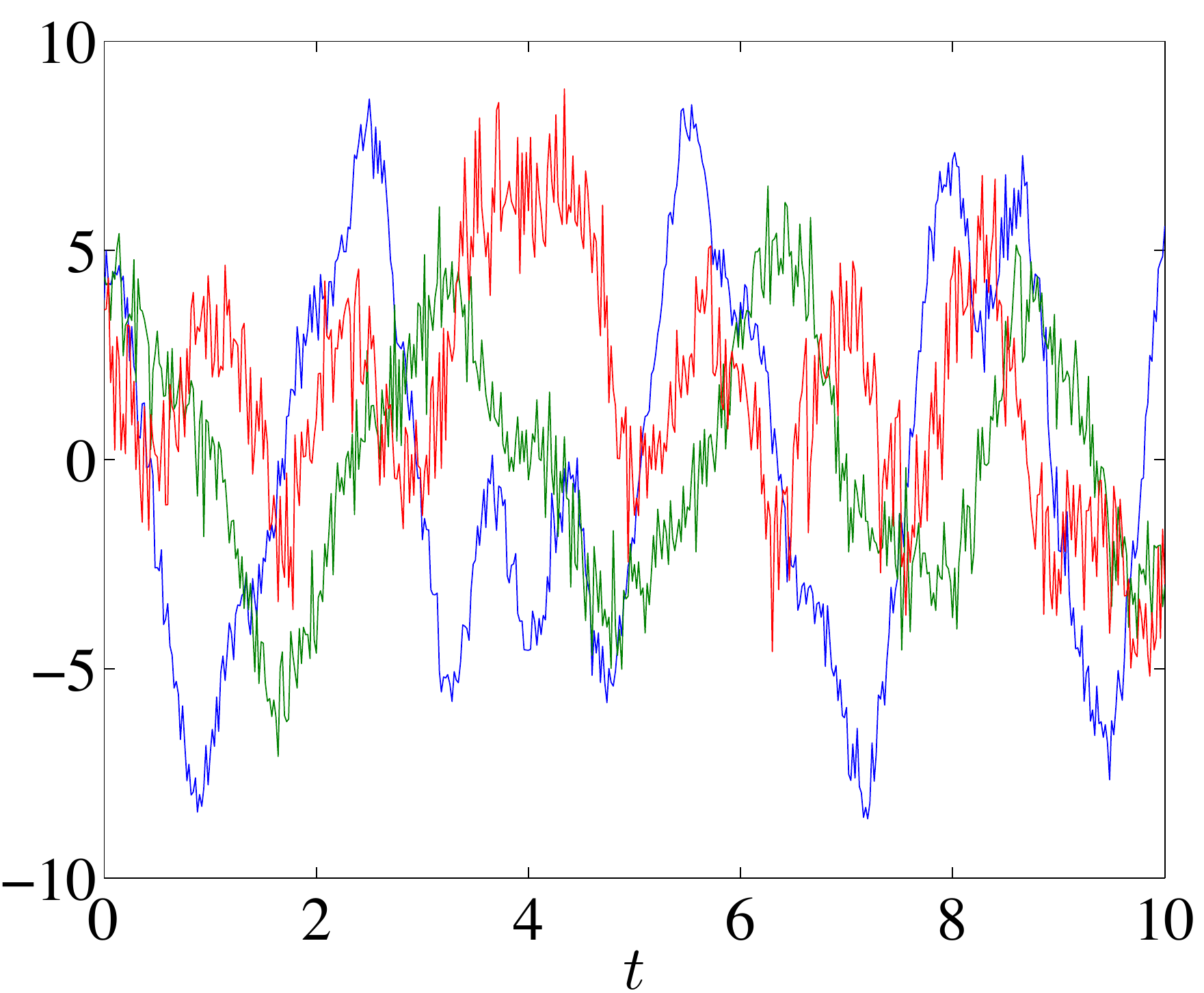}\label{fig:W_mea}}
}
\centerline{
	\subfigure[Visualization of $\mathcal{M}(F_z)$]{\hspace*{0.025\columnwidth}
		\includegraphics[width=0.35\columnwidth]{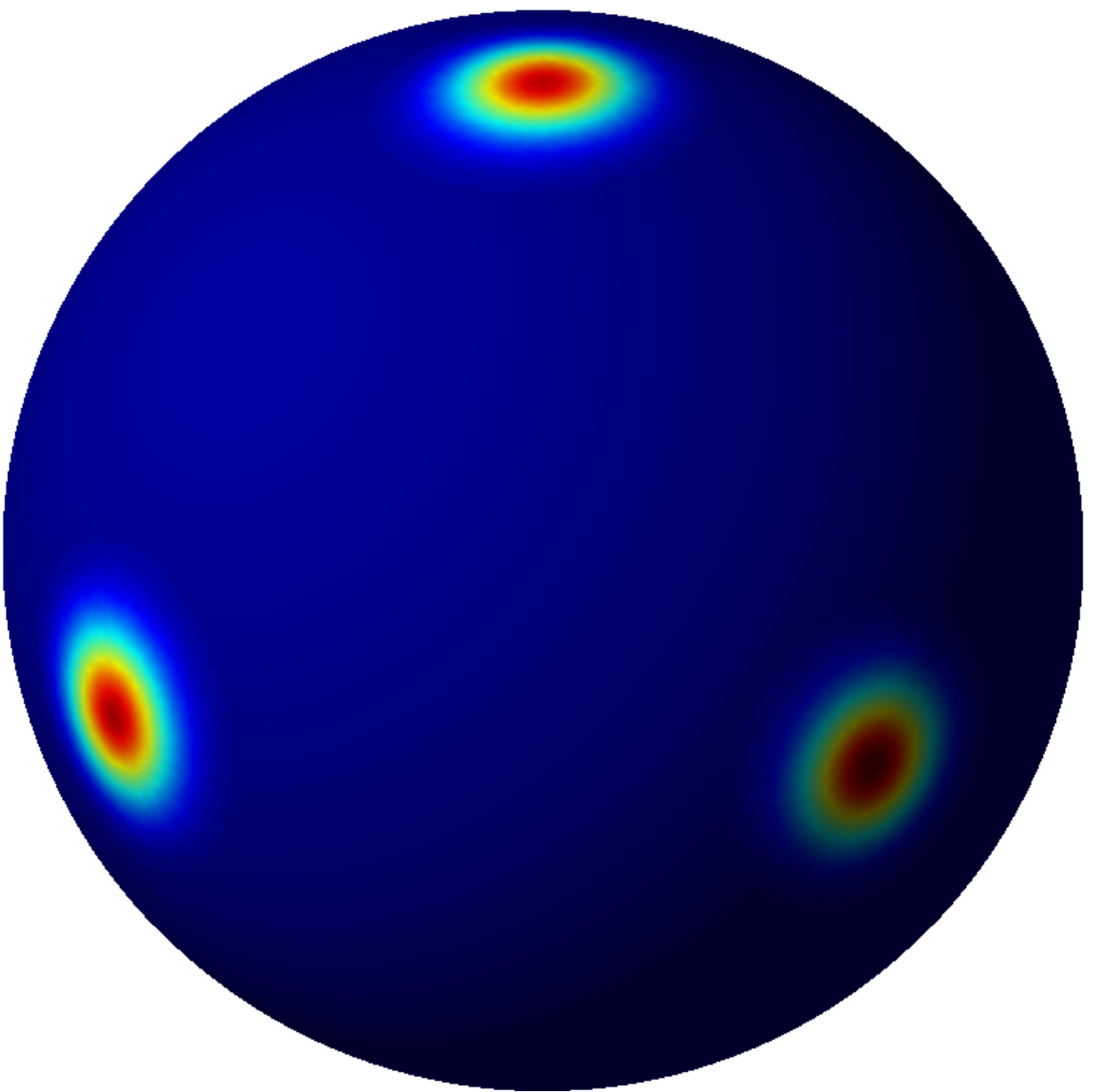}\label{fig:Fz}\hspace*{0.025\columnwidth}}
	\hfill
	\subfigure[Attitude measurement error (deg)]{
		\includegraphics[width=0.45\columnwidth]{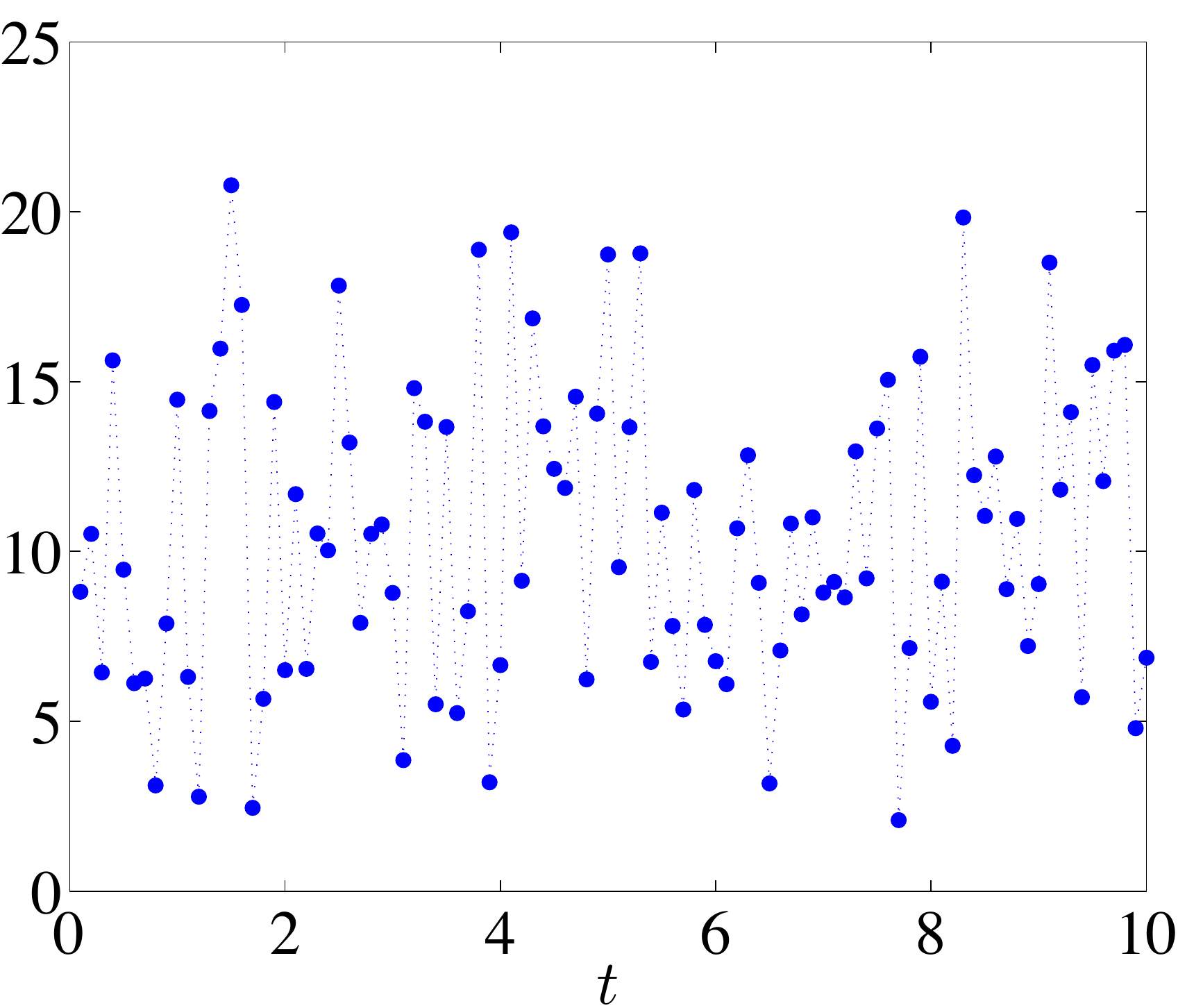}\label{fig:R_mea_err}}
}
\caption{True trajectory and measurement errors}\label{fig:true}
\end{figure}

We implement the proposed approach to a complex attitude dynamics for a 3D pendulum, which is a rigid body pendulum acting under a uniform gravity. It is shown that a 3D pendulum may exhibit highly irregular attitude maneuvers, and we adopt a particular nontrivial maneuver presented in~\cite{LeeChaPICDC07} as the \textit{true} attitude and angular velocity for the numerical example considered in this section. The initial true attitude and angular velocity are given by
\begin{align*}
R_{true}(0)= I,\quad \Omega_{true}(0)=[4.14,\,4.14,\,4.14]^T\,(\mathrm{rad/s}),
\end{align*}
and the resulting angular velocity trajectory is illustrated in \reffig{W_true}, which exhibits irregular rotational dynamics. 

It is assumed that the attitude and the angular velocity are measured at the rate of $10\,\mathrm{Hz}$ and $50\,\mathrm{Hz}$, respectively. The Fisher matrix for the attitude measurement error is chosen as $F_z= \mathrm{diag}[40,50,35]$, and the rotation matrix $W_R$ representing the attitude measurement error is sampled according to the rejection method described in~\cite{KenGan13}.  The matrix Fisher distribution for $F_z$, and the corresponding attitude measurement error for the sample  used in this numerical simulation are illustrated in \reffig{Fz} and \reffig{R_mea_err}, respectively. The mean attitude measurement error is $10.46^\circ$. The measurement error for the angular velocity is assumed to follow a normal distribution in $\Re^3$ with zero mean and the covariance matrix of $\mathrm{diag}[0.5^2,0.8^2,1^2]\,(\mathrm{rad/s})^2$. The angular velocity measurements are given in \reffig{W_mea}.

\subsection{Case I: Large initial estimate error}

We consider two cases depending on the estimate of the initial attitude. For Case I, the initial matrix parameter is
\begin{align*}
F(0) = 100 \exp (\pi\hat e_1),
\end{align*}
where the initial mean attitude is $M (0) = \exp (\pi\hat e_1)$, that corresponds to $180^\circ$ rotation of $R_{true}(0)$ about the first body-fixed axis. It is highly concentrated, since $S(0)=100I$ is large. In short, this represents the case where the estimator is falsely too confident about the incorrect attitude. 

The results of the proposed unscented attitude estimation are illustrated in \reffig{ES}, where the attitude estimation error is presented, and the degree of uncertainty in the estimates are measured via $\frac{1}{s_i}$. The estimation error rapidly reduces to $7.6^\circ$ from the initial error of $180^\circ$ after three attitude measurements at $t=0.3$, and the mean attitude error afterward is $5.6^\circ$. The uncertainties in the attitude increase until $t=0.3$ since the measurements strongly conflict with the initial estimate, but they decrease quickly after the attitude estimate converges. 

These can be also observed from the visualizations of $\mathcal{M}(F(t))$ for selected time instances in \reffig{visES}. Since the color shading of the figures is reinitialized in each figure, the value of the maximum probability density, corresponding to the dark red color, is specified as well. Initially, the probability distribution is highly concentrated, and it becomes dispersed a little at $t=0.08$ due to the angular velocity measurement error. But, after the initial attitude measurement is incorporated at $t=0.1$, the probability distributions for the second axis and the third axis become dispersed noticeably due to the conflict between the belief and the measurement. This is continued until $t=0.3$. But, later at $t=1$ and $t=10$, the estimated attitude distribution becomes concentrated about the true attitude.

\begin{figure}
\centerline{
	\subfigure[Attitude estimation error ($\mathrm{deg}$)]{\hspace*{0.02\columnwidth}
		\includegraphics[height=0.36\columnwidth]{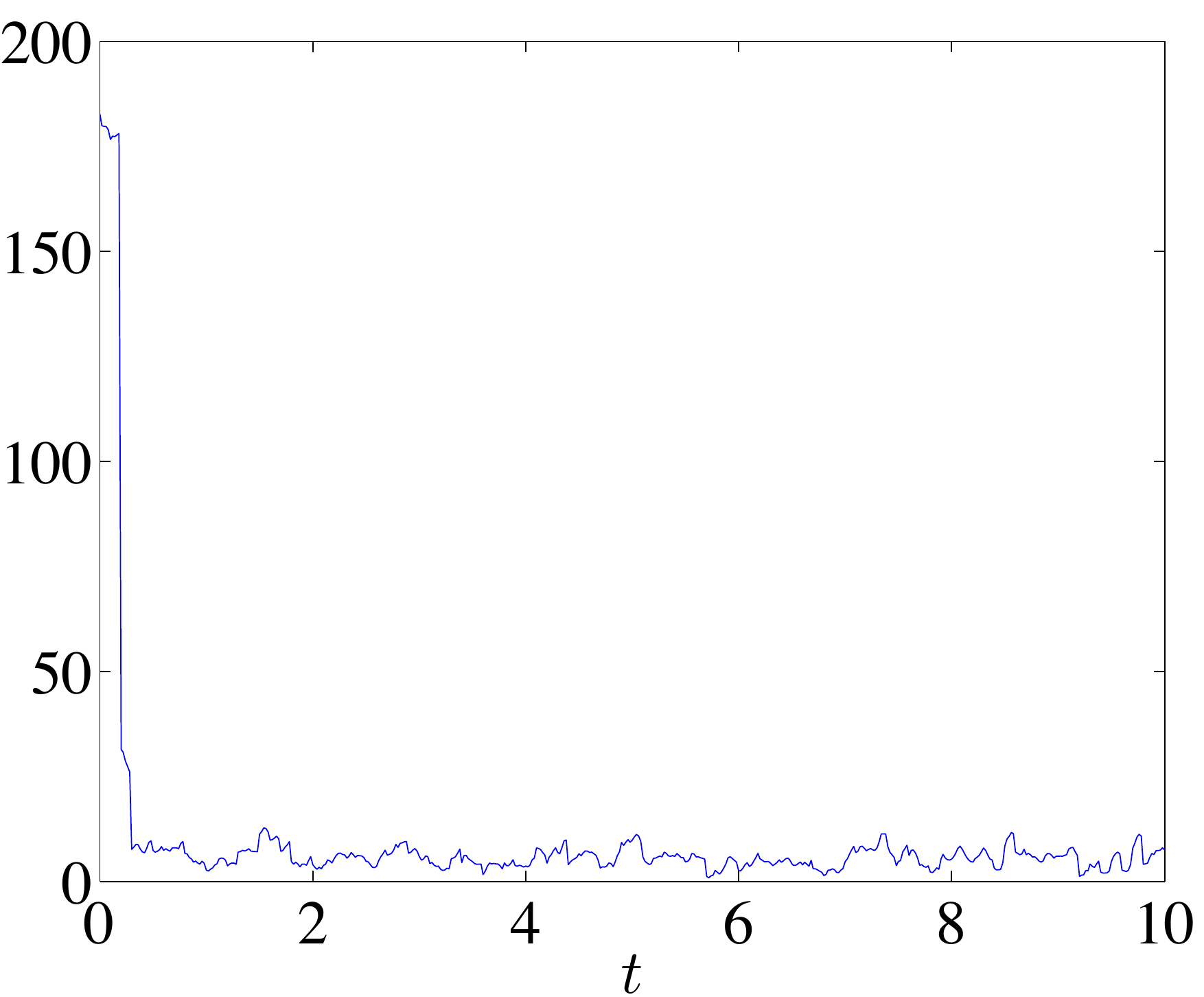}\label{fig:R_est_err}\hspace*{0.02\columnwidth}}
	\hfill
	\subfigure[Uncertainty measured by $1/s_i$]{\hspace*{0.02\columnwidth}
		\includegraphics[height=0.36\columnwidth]{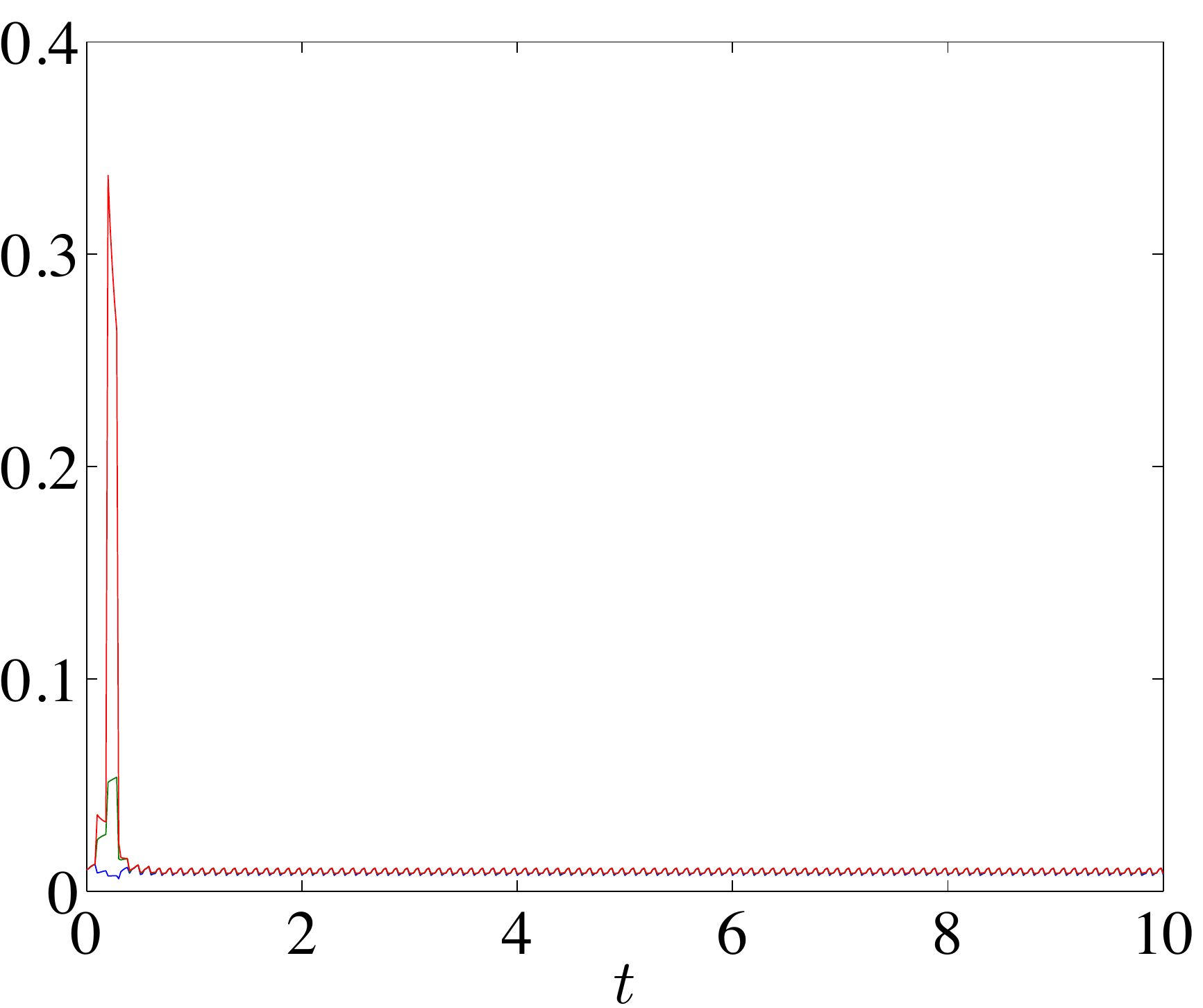}\label{fig:s}\hspace*{0.02\columnwidth}}
}
\caption{Case I: estimation results}\label{fig:ES}
\end{figure}
\begin{figure}
\centerline{
	\subfigure[$t=0$, $p_{\max}=1.41\times 10^4$]{\hspace*{0.06\columnwidth}
		\includegraphics[height=0.32\columnwidth]{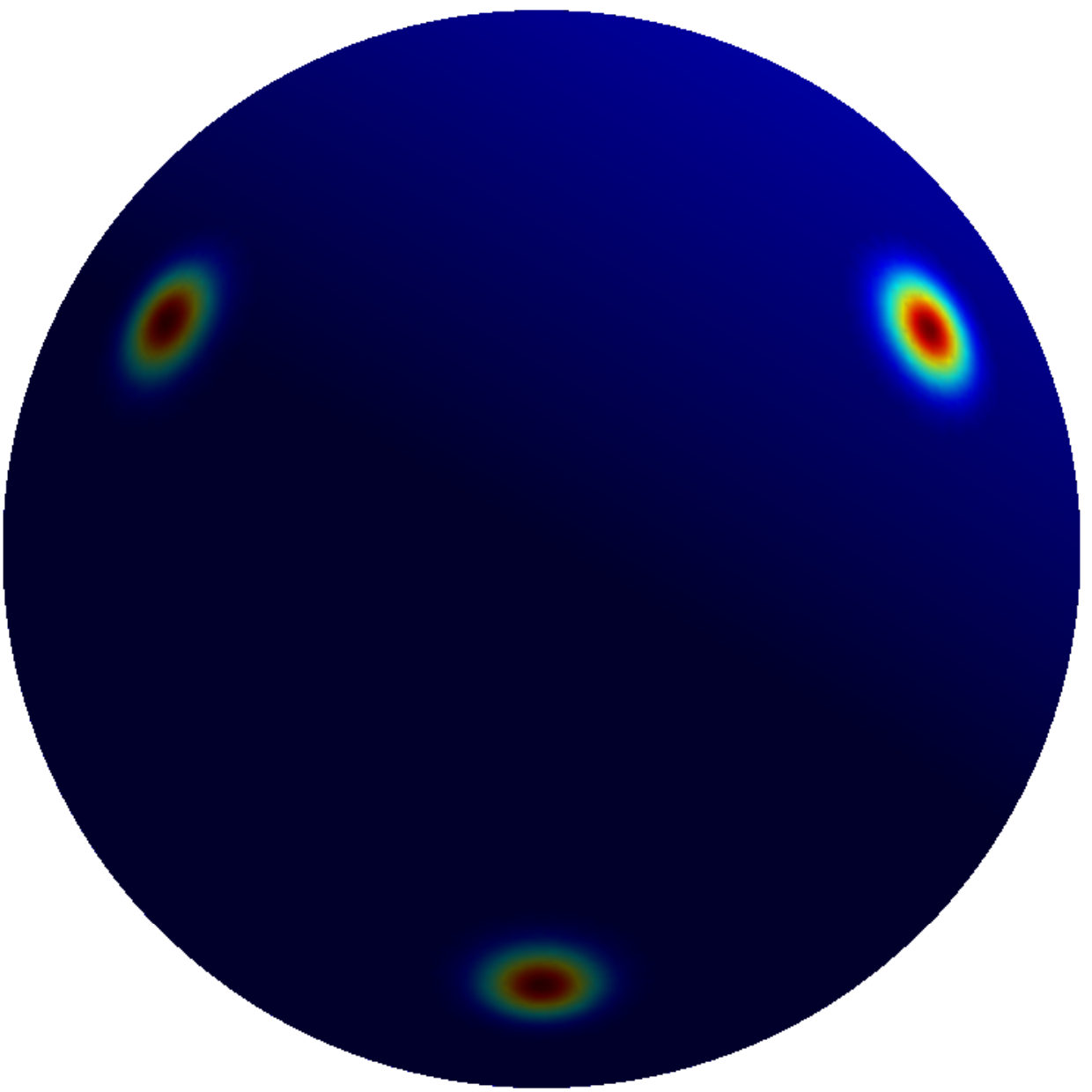}\label{fig:F_1}\hspace*{0.06\columnwidth}}
	\hfill
	\subfigure[$t=0.08$, $p_{\max}=9.92\times 10^3$]{\hspace*{0.06\columnwidth}
		\includegraphics[height=0.32\columnwidth]{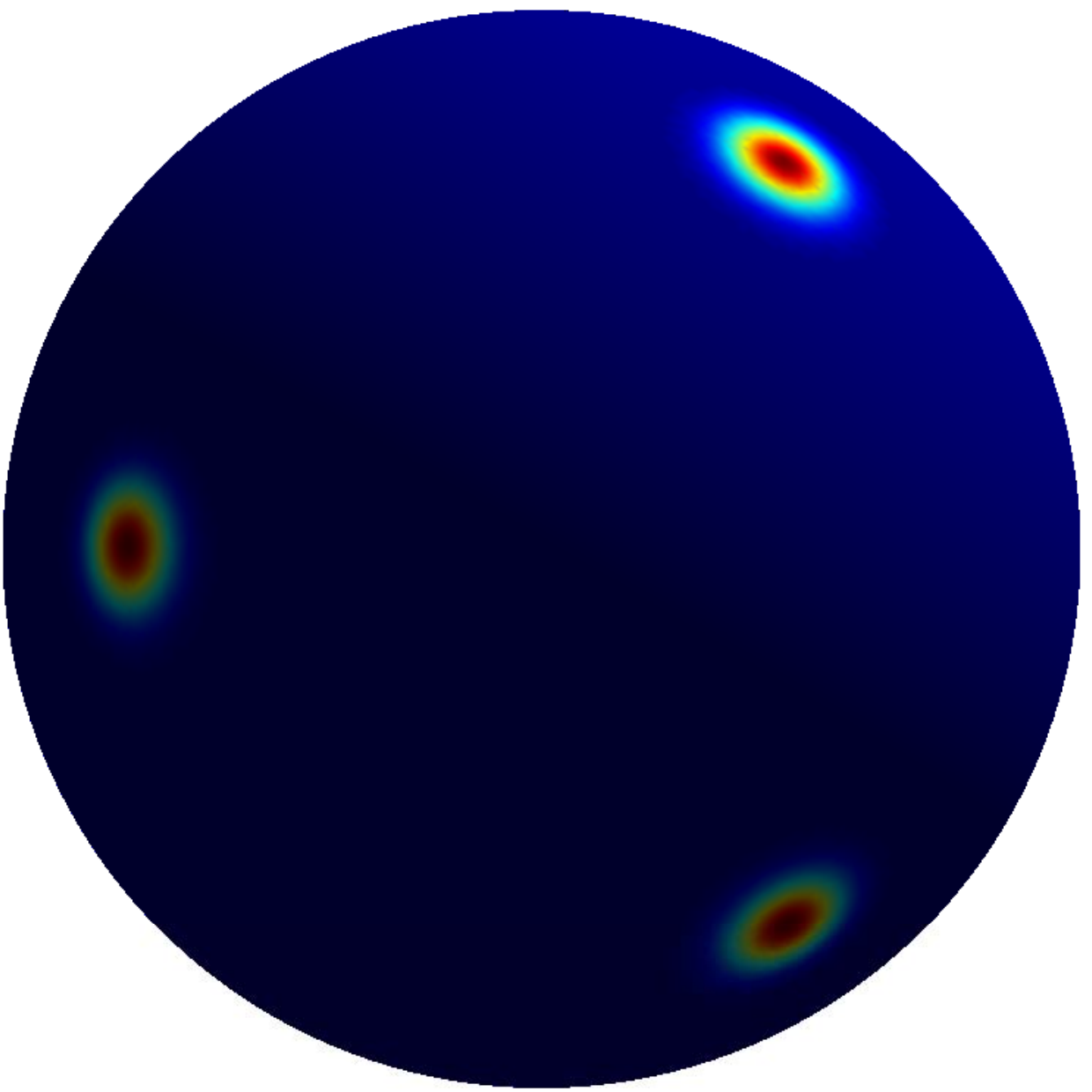}\label{fig:F_5}\hspace*{0.06\columnwidth}}
}
\centerline{
	\subfigure[$t=0.1$, $p_{\max}=6.27\times 10^3$]{\hspace*{0.06\columnwidth}
		\includegraphics[height=0.32\columnwidth]{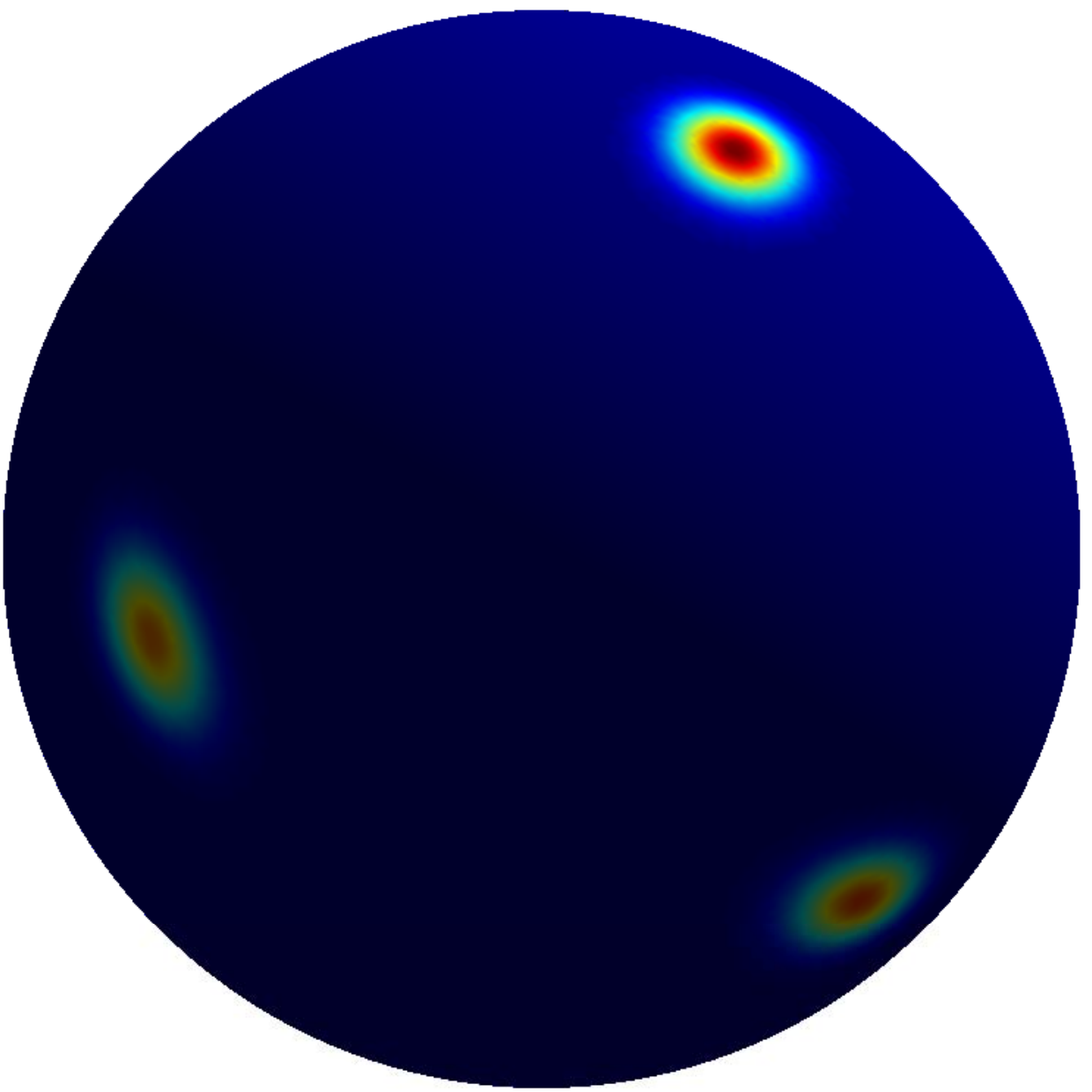}\label{fig:F_6}\hspace*{0.06\columnwidth}}
	\hfill
	\subfigure[$t=0.3$, $p_{\max}=1.18\times 10^4$]{\hspace*{0.06\columnwidth}
		\includegraphics[height=0.32\columnwidth]{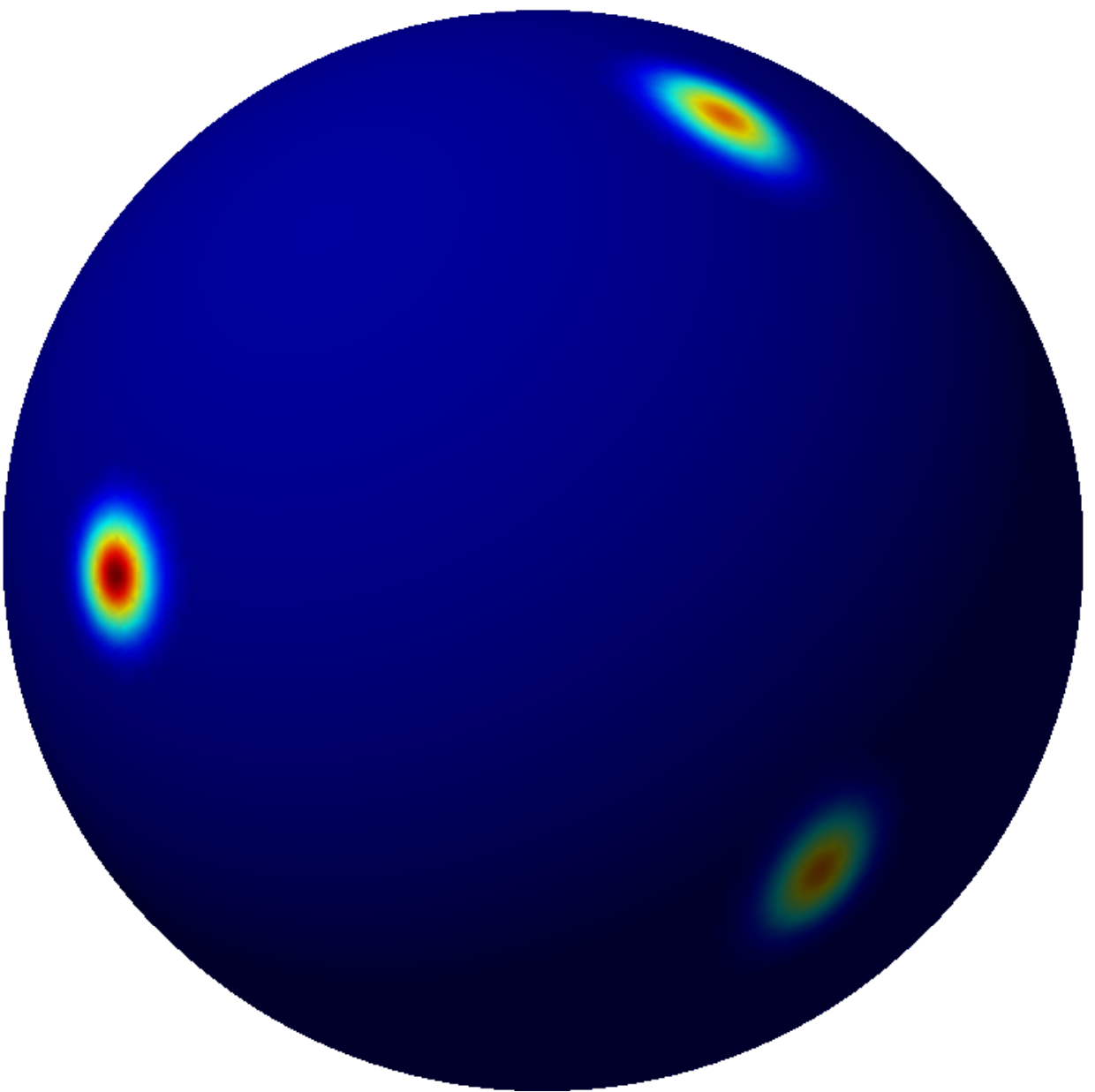}\label{fig:F_16}\hspace*{0.06\columnwidth}}
}
\centerline{
	\subfigure[$t=1$, $p_{\max}=2.00\times 10^4$]{\hspace*{0.06\columnwidth}
		\includegraphics[height=0.32\columnwidth]{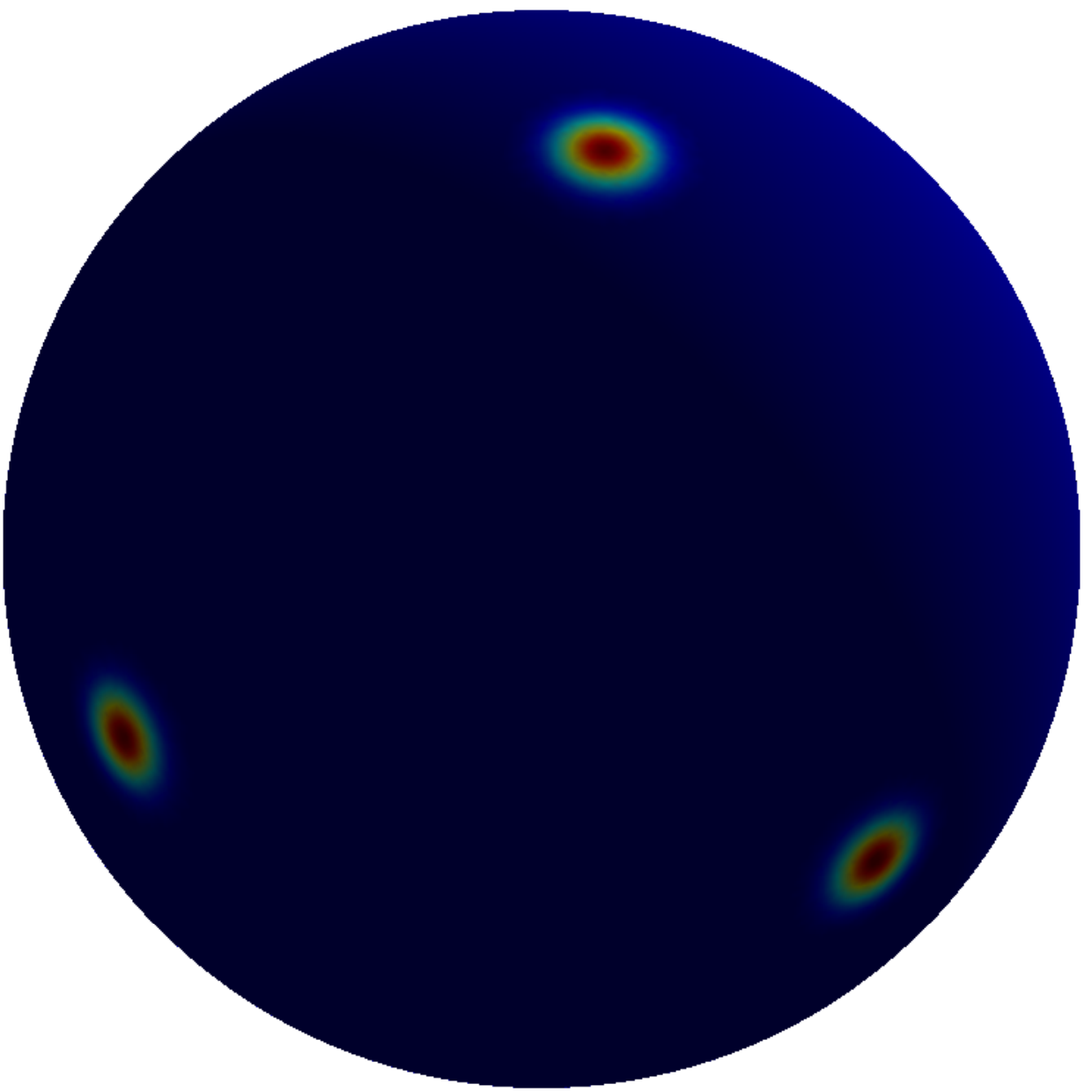}\label{fig:F_51}\hspace*{0.06\columnwidth}}
	\hfill
	\subfigure[$t=10$, $p_{\max}=2.02\times 10^4$]{\hspace*{0.06\columnwidth}
		\includegraphics[height=0.32\columnwidth]{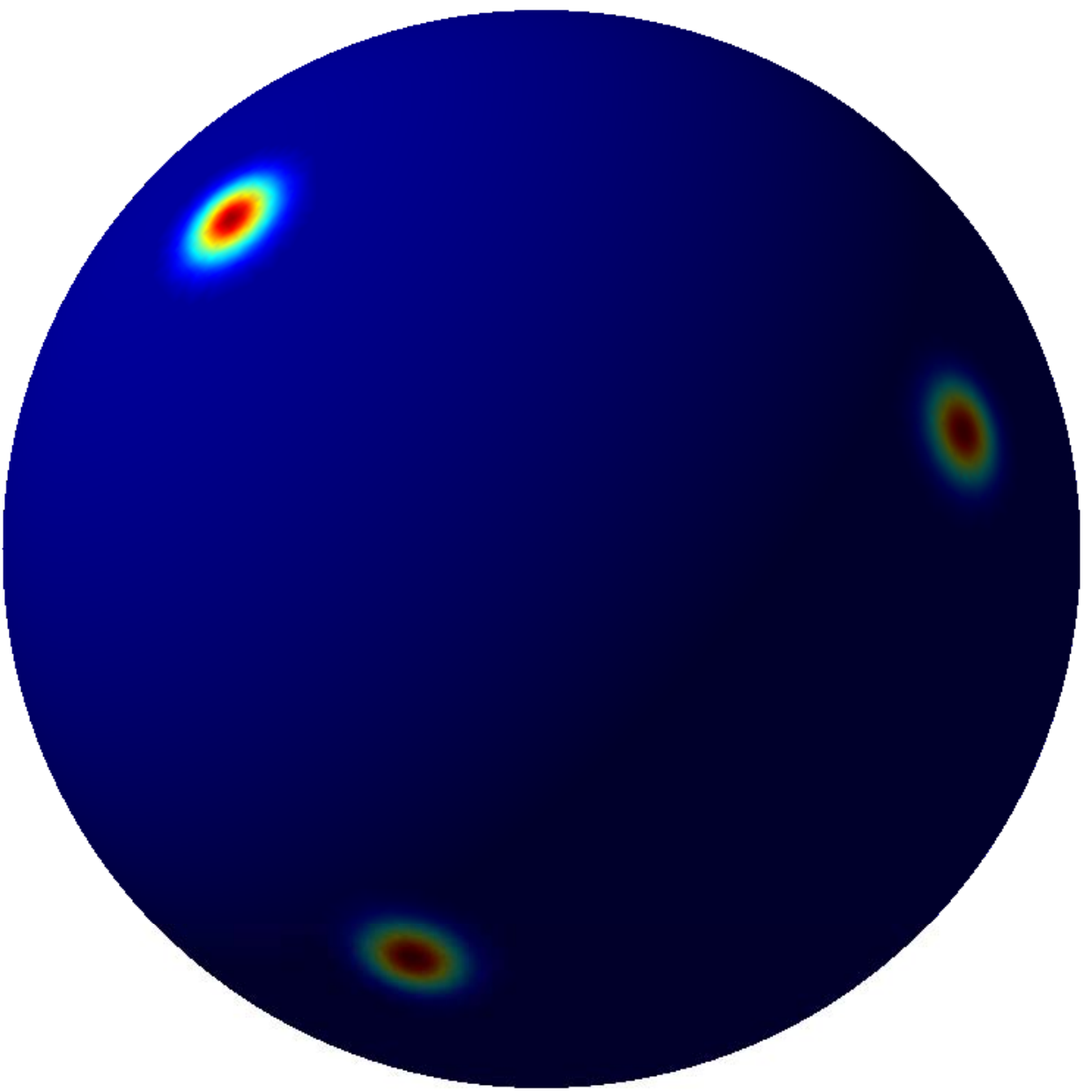}\label{fig:F_501}\hspace*{0.06\columnwidth}}
}
\caption{Case I: visualizations of $\mathcal{M}(F)$}\label{fig:visES}
\end{figure}

\subsection{Case II: Large initial uncertainty}

For the second case, the matrix parameter is chosen as
\begin{align*}
F(0)= \mathrm{diag}[2,1,0.5]\exp(0.5\pi\hat e_1),
\end{align*}
where the initial mean attitude has $90^\circ$ error, and it is largely diffused as $S(0)=\mathrm{diag}[2,1,0.5]$ is relatively small. This corresponds to the case with a large initial uncertainty. 

\begin{figure}
\centerline{
	\subfigure[Attitude estimation error ($\mathrm{deg}$)]{\hspace*{0.02\columnwidth}
		\includegraphics[height=0.36\columnwidth]{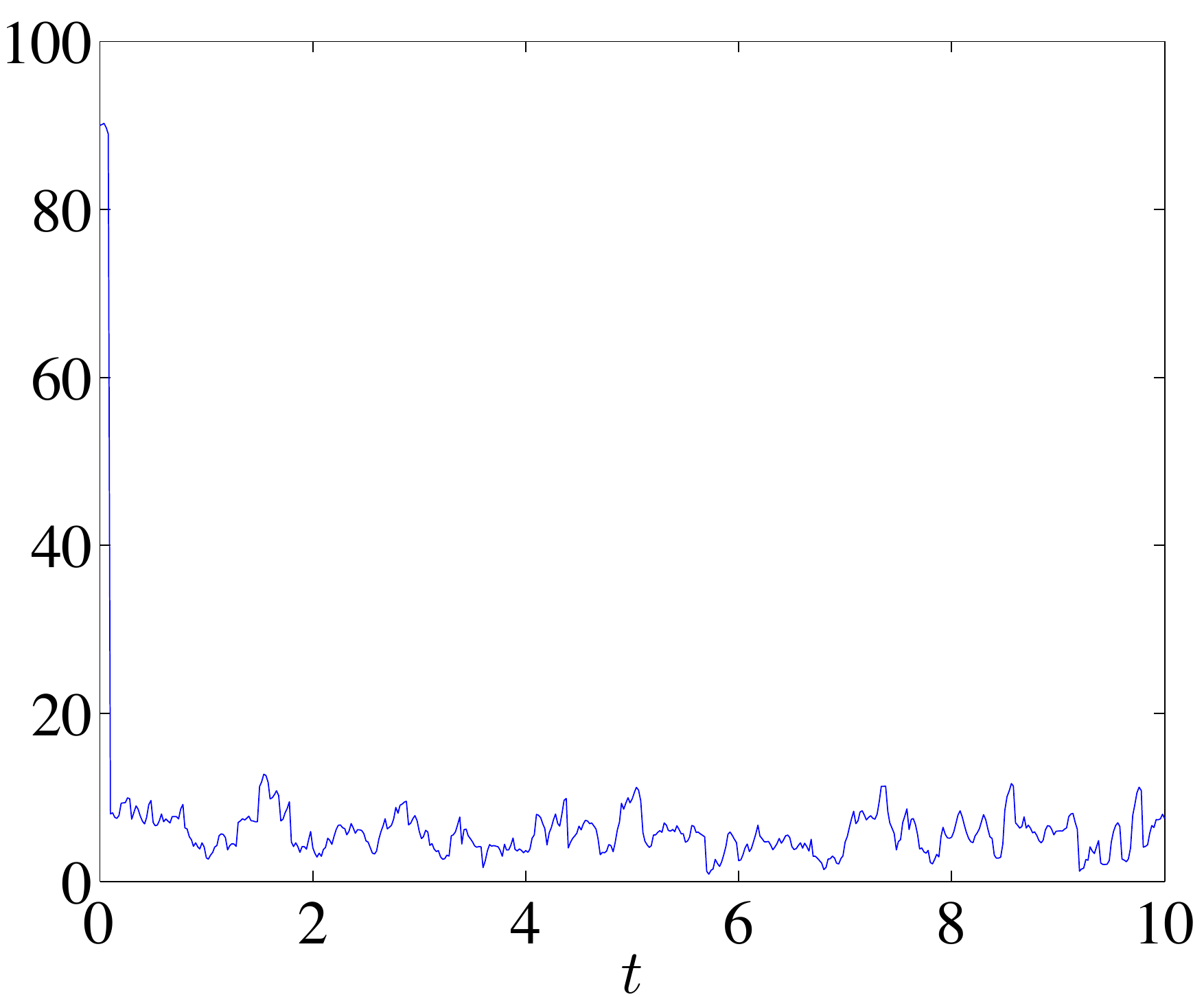}\label{fig:R_est_err_1}\hspace*{0.02\columnwidth}}
	\hfill
	\subfigure[Uncertainty measured by $1/s_i$]{\hspace*{0.02\columnwidth}
		\includegraphics[height=0.36\columnwidth]{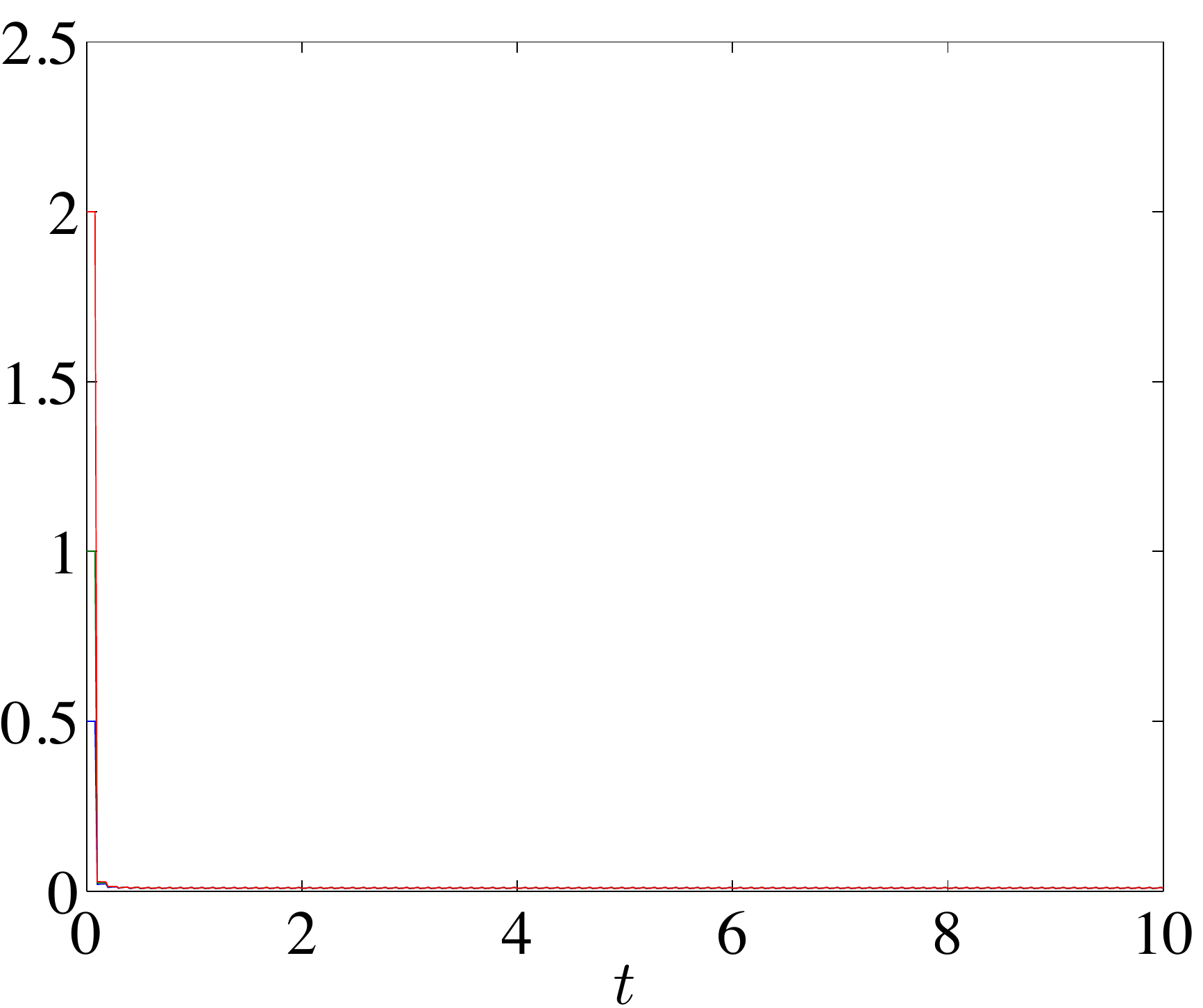}\label{fig:s_1}\hspace*{0.02\columnwidth}}
}
\caption{Case II: estimation results}\label{fig:ES_1}
\end{figure}
\begin{figure}
\vspace*{-0.1cm}
\centerline{
	\subfigure[$t=0$, $p_{\max}=1.30\times 10^1$]{\hspace*{0.06\columnwidth}
		\includegraphics[height=0.32\columnwidth]{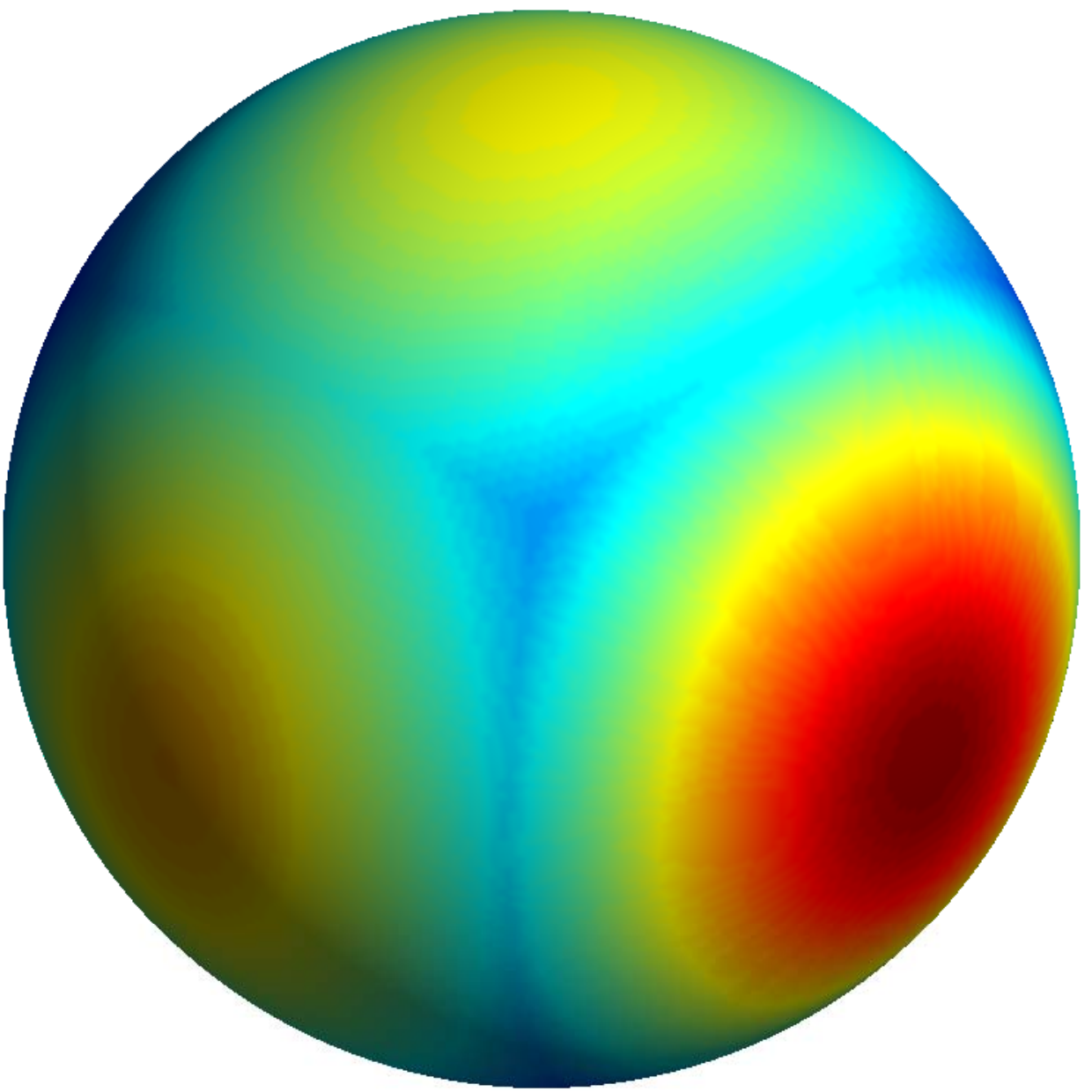}\label{fig:F_1_1}\hspace*{0.06\columnwidth}}
	\hfill
	\subfigure[$t=0.08$, $p_{\max}=1.30\times 10^1$]{\hspace*{0.06\columnwidth}
		\includegraphics[height=0.32\columnwidth]{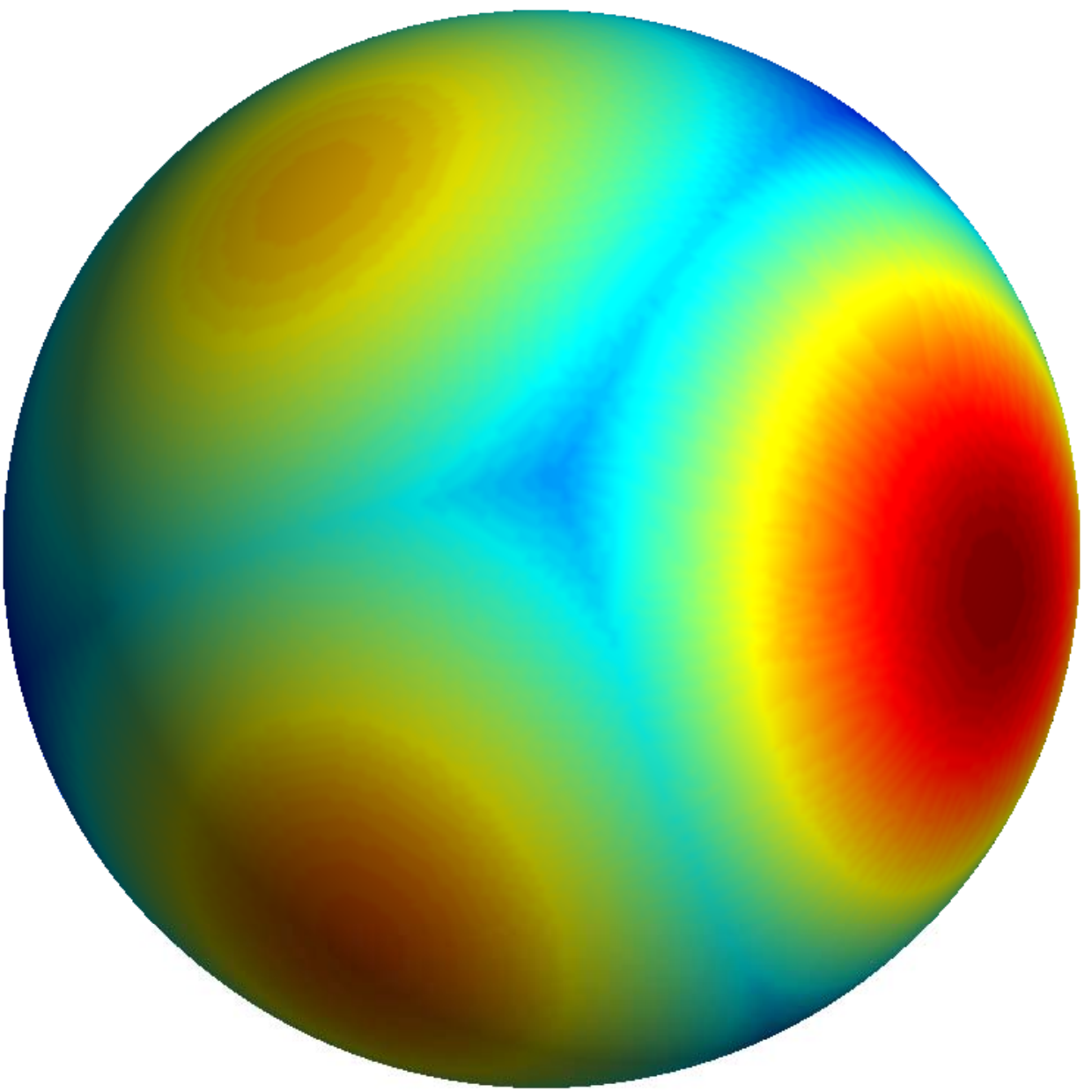}\label{fig:F_5_1}\hspace*{0.06\columnwidth}}
}
\centerline{
	\subfigure[$t=0.1$, $p_{\max}=3.86\times 10^3$]{\hspace*{0.06\columnwidth}
		\includegraphics[height=0.32\columnwidth]{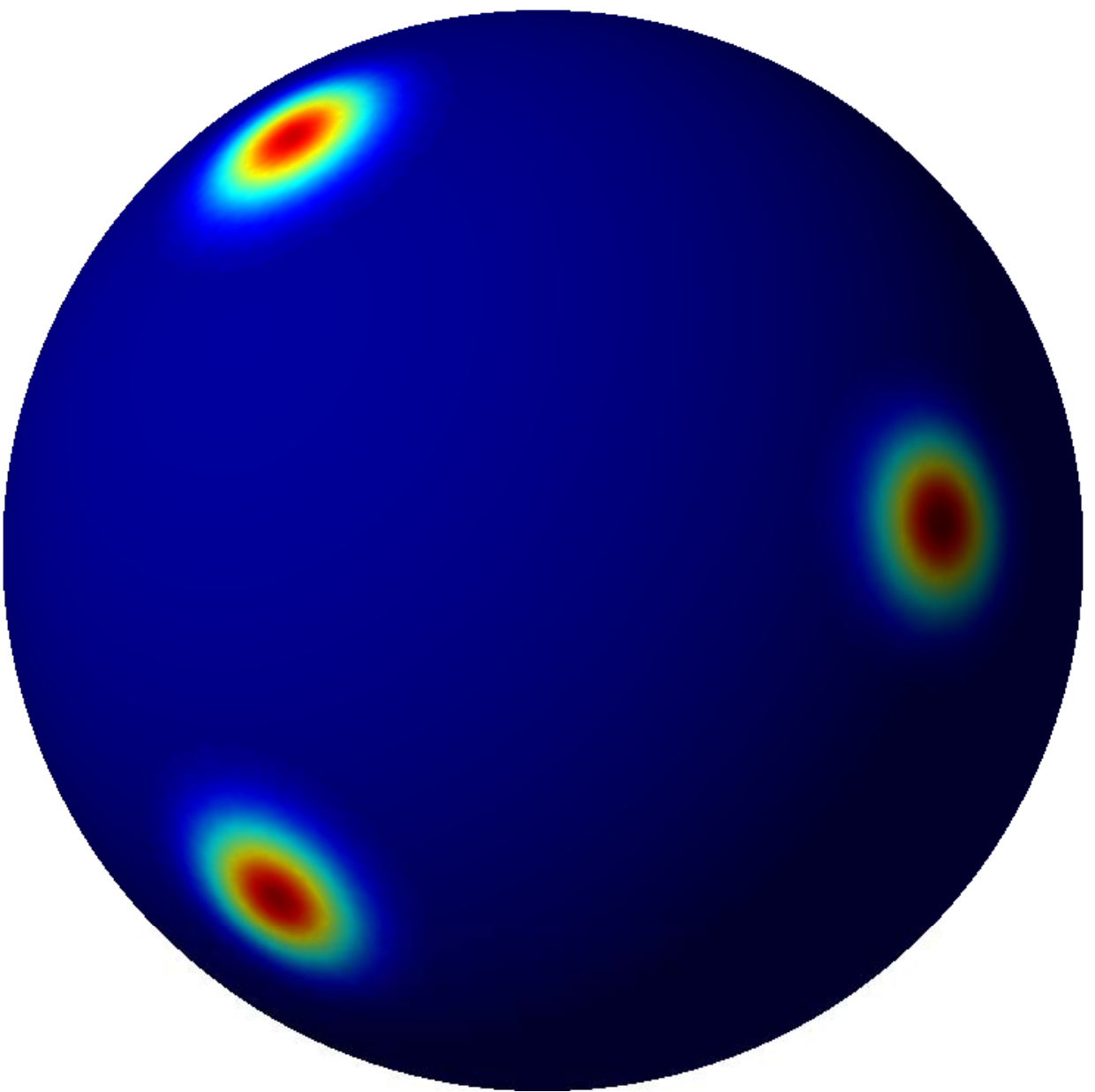}\label{fig:F_6_1}\hspace*{0.06\columnwidth}}
	\hfill
	\subfigure[$t=0.18$, $p_{\max}=3.71\times 10^3$]{\hspace*{0.06\columnwidth}
		\includegraphics[height=0.32\columnwidth]{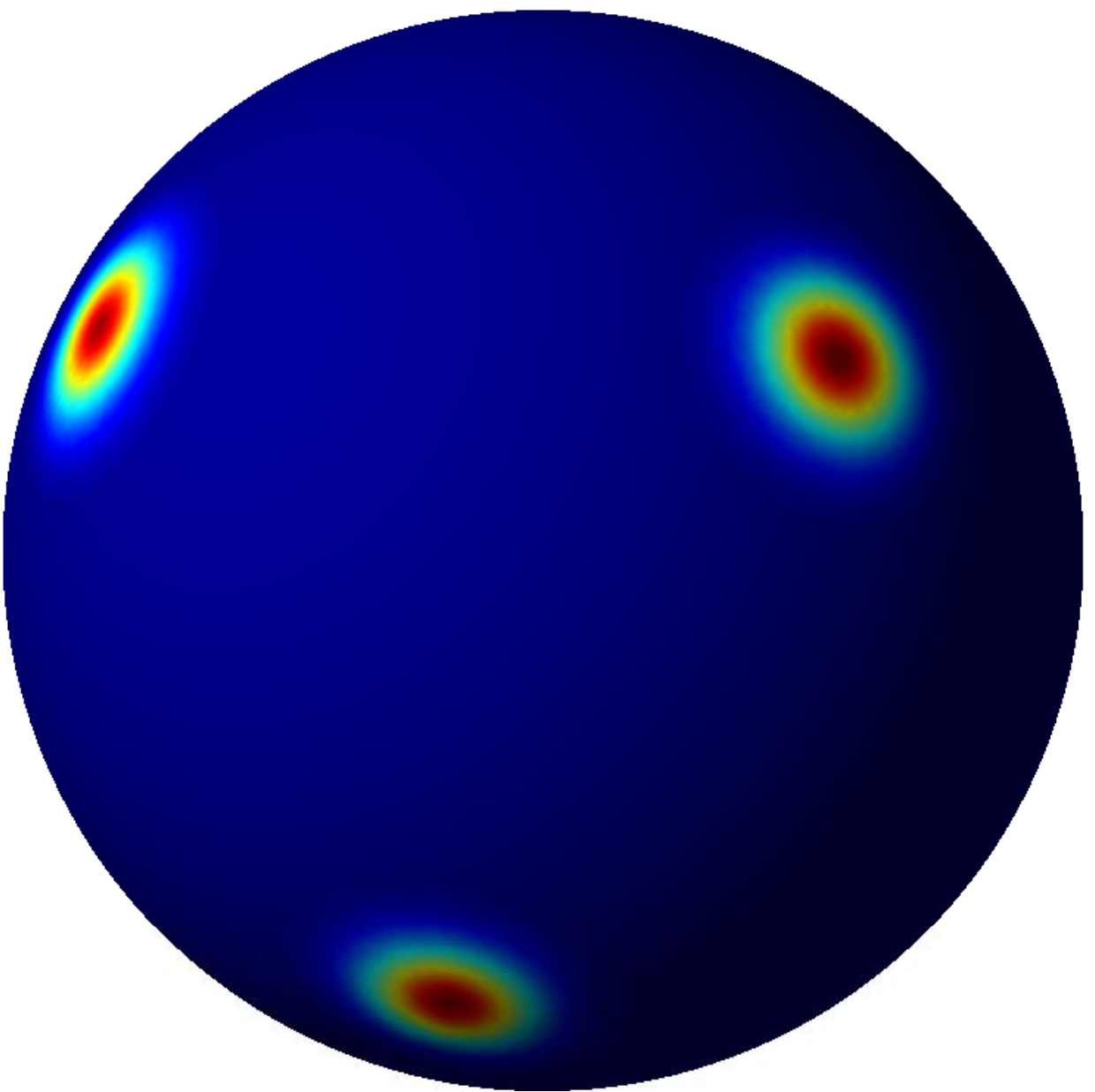}\label{fig:F_10_1}\hspace*{0.06\columnwidth}}
}
\centerline{
	\subfigure[$t=1$, $p_{\max}=1.05\times 10^4$]{\hspace*{0.06\columnwidth}
		\includegraphics[height=0.32\columnwidth]{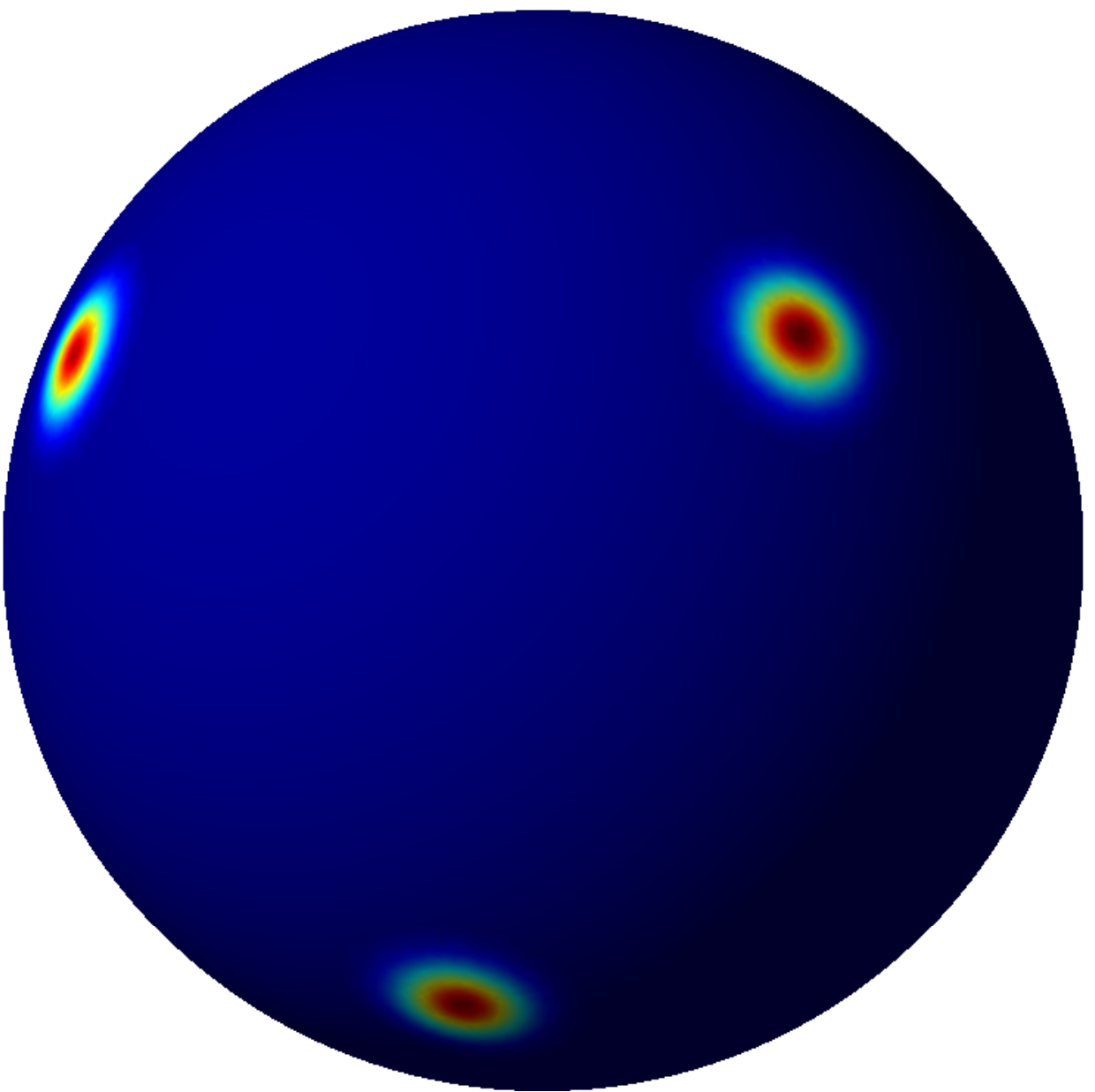}\label{fig:F_11_1}\hspace*{0.06\columnwidth}}
	\hfill
	\subfigure[$t=10$, $p_{\max}=2.02\times 10^4$]{\hspace*{0.06\columnwidth}
		\includegraphics[height=0.32\columnwidth]{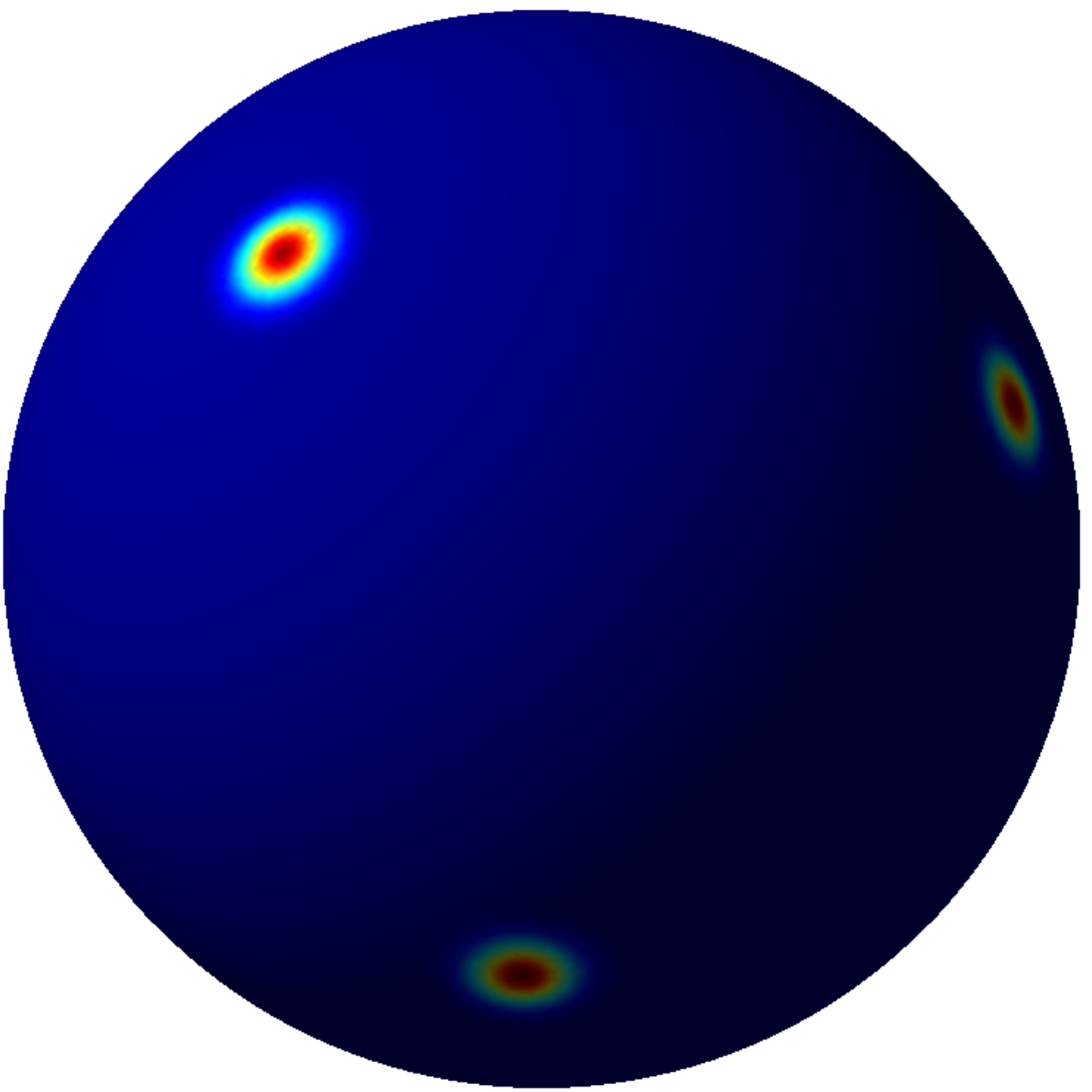}\label{fig:F_501_1}\hspace*{0.06\columnwidth}}
}
\caption{Case II: visualizations of $\mathcal{M}(F)$}\label{fig:visES_1}
\end{figure}

The corresponding numerical simulation results are presented in \reffig{ES_1} and \ref{fig:visES_1}. Both the attitude estimation error and the uncertainty decrease over time, since there is no strong conflict between the measurement and the estimate as opposed to the first case. In \reffig{visES_1}, it is illustrated that the estimated distribution becomes concentrated, especially after the first attitude measurement is received at $t=0.1$. 

The presented cases for attitude estimation are particularly challenging due to the following reasons: (i) the estimator is initially strongly confident about an incorrect attitude with the maximum error $180^\circ$, or the initial uncertainty is large; (ii) the considered attitude dynamics is swift and complex; (iii) both attitude and angular velocity measurement errors are relatively large; (iv) the attitude measurements are infrequent. These correspond to the cases where attitude estimators developed in terms of local coordinates or linearization tend to diverge. It is shown that the proposed approach developed directly on the special orthogonal group exhibits satisfactory, reasonable results even for the presented challenging cases.

\appendix

\subsection{Proof of Proposition 1}\label{sec:PfC}

First we show $c(F)=c(S)$. Substituting \refeqn{KM} into \refeqn{cF},
\begin{align*}
c(F) = \int_{\SO} \exp(\trs{SU^T R V}) dR. 
\end{align*}
Let $Q=U^T RV\in\SO$. The transformation from $R$ to $Q$ is volume-preserving, as $Q$ is obtained by multiplying rotation matrices, and $\SO$ is transformed into $\SO$ from the transform. We also have $dR=d(U^T RV)=dQ$ from the invariance of the Haar measure. Using these, we have
\begin{align}
c(F) = \int_{\SO} \exp(\trs{SQ}) dQ=c(S).  \label{eqn:cS0}
\end{align}

Next, we show \refeqn{cS}. Let $x\in\Sph^3$ be the quaternion corresponding to the rotation matrix $Q$, where the three-sphere is denoted by $\Sph^3=\{x\in\Re^4\,|\, \|x\|=1\}$. Let $q\in\Re^3$ and $q_4\in\Re$ be the vector part and the scalar part of the quaternion $x$, i.e., $x=[q^T,q_4]^T$.  It is well known that the corresponding rotation matrix is obtained by
\begin{align}
Q(x) = (q_4^2-q^Tq)I + 2qq^T + 2q_4\hat q, \label{eqn:Qx}
\end{align}
(see, for example, \cite{ShuJAS93}). Using several properties of the trace, 
\begin{align*}
\trs{S Q(x)} & = \trs{(q_4^2 - q^T q) S + 2qq^TS}
 = x^T B x,
\end{align*}
where the matrix $B\in\Re^{4\times 4}$ is given by
\begin{align}
B= \begin{bmatrix} 2S-\trs{S}I & 0_{3\times 1} \\ 0_{1\times 3} & \trs{S} \end{bmatrix}.\label{eqn:B}
\end{align}

Substituting this into \refeqn{cS0}, and by changing variables, 
\begin{align}
c(S) = \int_{\mathsf{RP}^3} \exp(x^T B x) \mathcal{J}(x) dx, \label{eqn:cS1}
\end{align}
where the real projective space, namely $\mathsf{RP}^3$ corresponds to $\Sph^3$ where the antipodal points are identified, and it is diffeomorphic to $\SO$ via \refeqn{Qx}, i.e., $Q(\mathsf{RP}^3)=\SO$. The scalar $\mathcal{J}(x)\in\Re$ is composed of two factors. The first one is to convert the three dimensional infinitesimal volume $dx$ on $\Sph^3$ to the three dimensional volume $Q(dx)$ on $\SO$, and the second factor accounts that $dx$ and $dR$ are normalized by the volume of $\Sph^3$ and $\SO$ respectively. 

In \cite{ShuJAS93}, the perturbation of $Q(x)$ is given by
\begin{align*}
(R^T \delta R)^\vee & = 2(q_4 \delta q - q \delta q_4 - q\times \delta q) = J(x) \delta x,
\end{align*}
where the matrix $J\in\Re^{3\times 4}$ is 
\begin{align*}
J(x)=2\begin{bmatrix} q_4 I-\hat q & -q \end{bmatrix}.
\end{align*}
Therefore, the scaling factor is 
\begin{align*}
\mathcal{J}(x) = \frac{2\pi^2}{8\pi^2}\sqrt{\mathrm{det}[J(x)J(x)^T]} = 
\frac{2\pi^2}{8\pi^2}\sqrt{\mathrm{det}[4I_{3\times 3}}]=2,
\end{align*}
where $2\pi^2$ and $8\pi^2$ correspond to the volume of $\Sph^3$ and $\SO$, respectively. Furthermore, the three-sphere can be considered as $\Sph^3 = \{x,-x\,|\, x\in\mathsf{RP}^3\}$, and $x^T B x$ is an even function of $B$. Applying these to \refeqn{cS1},
\begin{align}
c(S) &= \int_{\mathsf{RP}^3} 2\exp(x^T B x)\, dx 
=\int_{\Sph^3} \exp(x^T B x)\, dx.\label{eqn:cS2}
\end{align}

The above expression is equivalent to the normalizing constant of the Bingham distribution on $\Sph^3$~\cite{MarJup99}. The probability density of the Bingham distribution is given by
\begin{align*}
p_{\mathrm{Bing}}(x)=\frac{1}{b(B)} \exp (x^T B x),
\end{align*}
with respect to the uniform distribution, where $b(B)\in\Re$ is a normalizing constant, defined such that $b(B)=\int_{\Sph^3 }\exp (x^T B x) dx$. 
It has been shown that the normalizing constant $b(B)$ is given by the hypergeometric function of matrix argument, $b(B) = {}_1 F_1^{(2)} (\frac{1}{2},2;B)$~\cite{MarJup99}, which is shown to be evaluated as
\begin{align}
b(B) 
& = \int_{-1}^1 \frac{1}{2}I_0\!\bracket{\frac{1}{4}(b_2-b_1)(1-u)} I_0\!\bracket{\frac{1}{4}(b_4-b_3)(1+u)}\nonumber\\
&\times \exp\braces{-\frac{1}{2}(b_1+b_2)u}\,du,\label{eqn:b}
\end{align}
in~\cite{KunSchMG04,WooAJS93}, where $b_i$ denotes the $i$-th diagonal element of $B$, and $I_0$ denotes the modified Bessel function of the first kind. In short, the normalizing constant for the matrix Fisher distribution on $\SO$, corresponds to the normalizing constant for the Bingham distribution on $\Sph^3$, when the matrix $B$ is defined by \refeqn{B}, i.e., $c(F) = c(S) = b(B)$. The certain equivalence between the matrix Fisher distribution and the Bingham distribution was identified in~\cite{PreJRSSS86}, and an expression for $c(S)$ is presented in~\cite{WooAJS93} based on the relation. However, the reference did not consider the scaling factor $\mathcal{J}(x)$ properly.

Substituting \refeqn{B} into \refeqn{b}, we obtain \refeqn{cS} for the case when $(i,j,k)=(1,2,3)$. Any circular shift for the diagonal elements of $S$ can be written as $(C^T)^m S C^m$ for a positive integer $m$, where $C\in\SO$ is defined as $C=[e_3,e_1,e_2]$. 
According to the same argument to obtain \refeqn{cS0}, we have $c(S)= c((C^T)^m S C^m)$, i.e., the normalizing constant is invariant under any circular shifts of $s_i$. This shows \refeqn{cS}.

\subsection{Proof of Proposition 2}\label{sec:MD}

To derive the marginal distribution, we first consider the matrix Fisher distribution on $\mathsf{SO(2)}=\{R\in\Re^{2\times 2}\,|\, R^T R=I_{2\times 2},\,\mathrm{det}[R]=1\}$, given by
\begin{align*}
p_2(R) = \frac{1}{c_2(F)} \exp(\trs{F^T R}). 
\end{align*}
with respect to the uniform distribution on $\mathsf{SO(2)}$. Let the singular value decomposition of $F$ be given by $F=USV^T$ for $U,V\in\mathsf{SO(2)}$ and $S=\mathrm{diag}[s_1,s_2]$ for $s_1,s_2>0$. Similar to the proof of Proposition 1, the normalizing constant depends only on the singular values, i.e., $c_2(F)=c_2(S)$, and it is given by
\begin{align}
c_2(S) = \int_{\mathsf{SO(2)}} \exp(\trs{S^T R})dR. \label{eqn:c2S0}
\end{align}
We parameterize $\mathsf{SO}(2)$ via $\theta\in[0,2\pi)$ as 
\begin{align*}
R = \begin{bmatrix} \cos\theta & -\sin\theta \\ \sin\theta & \cos\theta \end{bmatrix}.
\end{align*}
Since $\trs{S^TR}=\trs{S}\cos\theta$, and $(R^T\delta R)^\vee = \delta \theta$,
\begin{align}
c_2(S) = \frac{1}{2\pi}\int_{0}^{2\pi}\exp (\trs{S}\cos\theta)d\theta= I_0(\trs{S}),\label{eqn:c2S}
\end{align}
where the factor $\frac{1}{2\pi}$ is included since $dR$ is normalized by the volume $2\pi$ of $\mathsf{SO}(2)$.

Next, we show \refeqn{pri}, \refeqn{c2}. To describe the proof more explicitly, we consider the case when $(i,j,k)=(1,2,3)$. For a given $r_1\in\Sph^2$, choose $r_{1c}\in\Re^{3\times 2}$ such that the columns of $r_{1c}$ span the orthogonal complement of $r_1$, and $[r_1,r_{1c}]\in\SO$. Then, any rotation matrix whose first column is $r_1$ can be written as $[r_1, r_{1c}Z]\in\SO$ for $Z\in\mathsf{SO(2)}$. The transformation from $R$ to $(r_1,Z)$ is shown to preserve the volume~\cite{KhaMarJRSSS77}. As such, the joint probability density for $r_1$ and $Z$ is written as
\begin{align*}
p(r_1,Z)& =\frac{1}{c(S)}\exp(\trs{F^T[r_1,r_{1c}Z]})\\
& = \frac{1}{c(S)} \exp (f_1^T r_1+\tr{ f_{23}^Tr_{1c} Z }).
\end{align*}
Integrating this with respect to $Z$ over $Z\in\mathsf{SO}(2)$, and using \refeqn{c2S0}, we obtain the marginal density for $r_1$ as
\begin{align*}
p(r_1) & = \frac{c_2(f_{23}^T r_{1c})}{c(S)}   \exp (f_1^Tr_1).
\end{align*}
From \refeqn{c2S}, $c_2(f_{23}^T r_{1c})$ depends only on the sum of two singular values of $f_{23}^T r_{1c}\in\Re^{2\times 2}$. Using the fact that $r_{1c}r_{1c}^T=I_{3\times 3}-r_1r_1^T$, we obtain \refeqn{c2} when $(i,j,k)=(1,2,3)$.
%
%
Other cases for $(i,j,k)\in\mathcal{I}$ can be shown similarly. 

\subsection{Proof of Proposition 3}\label{sec:UT}

From \refeqn{KM} and \refeqn{Ri}, the arithmetic mean can be written as
\begin{align}
\bar R & = \frac{1}{7} U\bracket{I_{3\times 3} + \sum_{i=1}^3 \{\exp(\theta_i\hat e_i) + \exp(-\theta_i\hat e_i)\}}V^T.\label{eqn:Rbar0}
\end{align}
Using Rodriguez' formula~\cite{ShuJAS93},
\begin{align*}
\exp(\theta_i\hat e_i) + \exp(-\theta_i\hat e_i) 
& = 2(\cos\theta_i I +(1-\cos\theta_i)e_ie_i^T),
\end{align*}
which is a diagonal matrix where the $i$-th diagonal elements is 2, and the other diagonal elements are $2\cos\theta_i$. Therefore, the expression in the bracket of \refeqn{Rbar0} reduces to the matrix $D$ at \refeqn{UDV}. This shows \refeqn{UDV}.

\bibliography{/Users/tylee/Documents/BibMaster}
\bibliographystyle{IEEEtran}

\end{document}